\documentclass[12pt,a4paper]{amsart}
\usepackage{amsfonts}
\usepackage{amssymb}
\usepackage{amsmath}
\usepackage{fullpage}
\usepackage{graphicx}

\RequirePackage{amsthm}

\newtheorem{thm}{Theorem}[section]
\newtheorem{cor}[thm]{Corollary}

\newtheorem{prop}[thm]{Proposition}
\theoremstyle{definition}
\newtheorem{rem}[thm]{Remark}

\numberwithin{equation}{section}

\begin{document}
\title[Positive solutions of $(p,q)$-Laplacian systems]{A topological approach to the existence and multiplicity of positive solutions of 
$(p,q)$-Laplacian systems}
\author[G. Infante]{Gennaro Infante}
\address{Gennaro Infante, Dipartimento di Matematica ed Informatica,
Universit\`{a} della Calabria, 87036 Arcavacata di Rende, Cosenza, Italy}
\email{gennaro.infante@unical.it}
\author[M. Maciejewski]{Mateusz Maciejewski}
\address{Mateusz Maciejewski, Nicolaus Copernicus University, Faculty of
Mathematics and Computer Science, ul. Chopina 12/18, 87-100 Toru\'n, Poland}
\email{Mateusz.Maciejewski@mat.umk.pl}
\author[R. Precup]{Radu Precup}
\address{Radu Precup, Departamentul de Matematic\u{a}, Universitatea Babe%
\c{s}-Bolyai, Cluj 400084, Romania}
\email{r.precup@math.ubbcluj.ro}
\subjclass[2010]{Primary 35J57, secondary 35B09, 35B45, 35D30, 35J92, 47H10}
\keywords{Weak Harnack inequality, fixed point index, $p$-Laplace operator,
quasilinear elliptic system, positive weak solution, cone, multiplicity,
nonexistence.}

\begin{abstract}
In this paper we develop a new theory for the existence, localization and
multiplicity of positive solutions for a class of non-variational,
quasilinear, elliptic systems. In order to do this, we provide a fairly
general abstract framework for the existence of fixed points of nonlinear
operators acting on cones that satisfy an inequality of Harnack type. Our
methodology relies on fixed point index theory. We also provide a
non-existence result and an example to illustrate the theory.
\end{abstract}

\maketitle

\section{Introduction}

In this paper we develop a new theory for the existence, localization and
multiplicity of nonnegative weak solutions to the following Dirichlet
problem for $(p,q)$-Laplacian systems%
\begin{equation}
\left\{
\begin{array}{ll}
-\Delta _{p}u=f(x,u,v) & \text{in }\Omega , \\
-\Delta _{q}v=g(x,u,v) & \text{in }\Omega , \\
u,v=0 & \text{on }\partial \Omega ,%
\end{array}%
\right.  \label{eq:diff}
\end{equation}
where $\Omega \subset \mathbb{R}^{n}$ is a bounded domain of class $%
C^{1,\gamma }$ for some $\gamma \in (0,1),$ $f,g:\overline{\Omega }\times
\mathbb{R}_{+}^{2}\rightarrow \mathbb{R}_{+}$ are continuous functions, and $%
p,q>2n/\left( n+1\right) .$

The existence of positive solutions of these types of problems has been
investigated by means of different methodologies. The most common approach
is topological, for example Schauder fixed point theorem was used in~\cite%
{MR2302915}, Schaefer fixed point theorem was employed in~\cite{MR2598811},
Leray-Shauder degree theory was exploited in~\cite{MR1781264, MR2233375,
MR3041022}, fixed point index on cones was applied in ~\cite{MR2927480,
MR2648206} and continuation methods were used in~\cite{MR1929884}.
Variational methods, making use of the Nehari manifold, have been used in
\cite{MR2946850,MR2746785}, super and subsolution methods were applied in ~%
\cite{MR2945973, MR2038734} and monotone techniques in~\cite{MR2807411}.

The multiplicity of solutions was studied in \cite{MR2946850,MR2746785,
MR3041022}, nonexistence was investigated in \cite{MR2233375, MR2648206,
MR3041022}, a priori estimates were given in \cite{MR1929884, MR3041022},
regularity results were obtained in \cite{MR2945973} and qualitative
properties of the solutions have been studied in \cite{MR2683690}.

Localization results have been given in \cite{MR2927480}, where the authors
proved the existence of one positive solution in the case $p=q,$ and in \cite%
{MR1781264}, where the authors dealt with the existence of radially
symmetric solutions in a ball.

Here we develop a new method that deals with the existence, localization and
multiplicity of solutions in cones, for systems of two (or more) abstract
equations. This method can be applied, \emph{as a special case}, to deal
with the existence of positive solutions of the system~\eqref{eq:diff}. Our
approach is purely topological and is based on an abstract Harnack-type
inequality and on the fixed point index theory. In order to do this, we
fully exploit the recent theory developed by Precup in~\cite{b:RP_TMNA, b:RP
JFPTAA} for the case of one equation, where the author has obtained some
Krasnosel'skii-type results. We remark that our results are not a trivial
extension of the ones in~\cite{b:RP_TMNA, b:RP JFPTAA} to the case of
systems. In fact, we fully benefit of the richer structure of the system and
we improve the theory, even in the case of one equation, by allowing better
constants and a more precise localization of the solutions.

In the case of the system~\eqref{eq:diff} we obtain existence, localization,
multiplicity and non-existence of positive weak solutions. We also provide
an example to illustrate the theoretical results.

Our results are new and improve and complement earlier ones in literature.

\section{Operator equations on Cartesian products}

\label{sect:abstract}

Let $X_{i}\ (i=1,2)$ be Banach spaces with norms $|\cdot|_{i}$ ordered by cones $%
K_{i}^{0}$ and let $\Vert \cdot \Vert _{i}$ be seminorms on $X_{i}.$ Denote by $%
\leq _{i}$ the partial order relation associated with $K_{i}^{0}.$ Assume
that both norms and seminorms are monotone, i.e.%
\begin{equation*}
0\leq _{i}u\leq _{i}v\ \ \text{implies\ \ }|u|_{i}\leq |v|_{i}\ \ \text{and\
\ }\Vert u\Vert _{i}\leq \Vert v\Vert _{i}
\end{equation*}%
for $u,v\in X_{i}.$ In what follows, for simplicity, we shall use the same
symbols $\left\vert \cdot\right\vert ,\ \left\Vert \cdot \right\Vert   ,\ \leq $ $\ $to
denote $\left\vert \cdot\right\vert _{i},\ \left\Vert \cdot \right\Vert   _{i},$ $\leq
_{i}$ for both $i=1,2.$

Let $\chi _{i}\in K^0_i$ be fixed such that $\Vert \chi _{i}\Vert >0.$
Define the cones $K_{i}\subset K^0_i$ by the formula
\begin{equation*}
K_{i}:=\left\{ u\in K^0_i:u\geq \left\Vert u\right\Vert \chi _{i}\right\}
\end{equation*}%
and assume that there exist points inside them with positive seminorms,
which is equivalent to the assumption $\Vert \chi _{i}\Vert \leq 1.$ Hence,
we can choose
\begin{equation}  \label{eq:phi_i}
\phi _{i}\in K_{i},\ \ |\phi _{i}|=1,\ \ \Vert \phi _{i}\Vert >0.
\end{equation}
In particular we may take%
\begin{equation*}
\phi _{i}=\frac{\chi _{i}}{\left\vert \chi _{i}\right\vert }.
\end{equation*}%
Let us observe that the seminorm $\|\cdot\|_i $ is continuous in $K_i$ with
respect to the topology induced by the norm $|\cdot|_i$, since one has:
\begin{equation}  \label{eq:seminorm-cont}
u\geq \|u\| \chi_i,\ u\in K_i \implies \|u\|\leq \frac 1{|\chi_i|}|u|,\ u\in
K_i.
\end{equation}

In what follows by \emph{compactness} of a continuous operator we mean the
relative compactness of its range. By \emph{complete continuity} of a
continuous operator we mean the relative compactness of the image of every
bounded set of the domain.

We set
\begin{equation*}
X:=X_{1}\times X_{2},\quad K:=K_{1}\times K_{2},
\end{equation*}%
and we seek the fixed points of a completely continuous operator
\begin{equation*}
N=(N_{1},N_{2})\colon K\rightarrow K,
\end{equation*}%
that is $\left( u,v\right) \in K$ with $N\left( u,v\right) =\left(
u,v\right) .$ Note that the cone-invariance of $N$ is equivalent to the fact
that the operators $N_{i}$ satisfy an \emph{abstract weak Harnack inequality}
of the type
\begin{equation*}
N_{i}(u,v)\geq \Vert N_{i}(u,v)\Vert \chi _{i},\ \ \ \left( i=1,2\right) .
\end{equation*}

We shall discuss not only existence, but also localization and multiplicity
of solutions of the nonlinear equation $N\left( u,v\right) =\left(
u,v\right).$ In order to do this, we utilize the Granas fixed point index, $%
\mathrm{ind}_C(f,U)$ (for more information on the index and its applications
we refer the reader to \cite{d, gd}).

The next Proposition describes some of the useful properties of the index,
for details see Theorem 6.2, Chapter 12 of \cite{gd}.

\begin{prop}
\label{fpi-pro} Let $C$ be a closed convex subset of a Banach space, $%
U\subset C$ be open in $C$ and $f\colon \overline{U}\rightarrow C$ be a
compact map with no fixed points on the boundary $\partial U$ of $U.$ Then
the fixed point index has the following properties:

\emph{(i)} (Existence) If $\mathrm{ind}_{C}(f,U)\neq 0$ then $\mathrm{fix}%
(f)\neq \emptyset $, where $\mathrm{fix}f=%
\mbox{$\left\{x\in\overline
U:\;f(x)=x\right\}$}$.

\emph{(ii)} (Additivity) If $\mathrm{fix}f\subset U_{1}\cup U_{2}\subset U$
with $U_{1},U_{2}$ open in $C$ and disjoint, then
\begin{equation*}
\mathrm{ind}_{C}(f,U)=\mathrm{ind}_{C}(f,U_{1})+\mathrm{ind}_{C}(f,U_{2}).
\end{equation*}

\emph{(iii)} (Homotopy invariance) If $h\colon \overline{U}\times \lbrack
0,1]\rightarrow C$ is a compact homotopy such that $h(u,t)\neq u$ for $u\in
\partial U$ and $t\in \lbrack 0,1]$ then
\begin{equation*}
\mathrm{ind}_{C}(h(.,0),U)=\mathrm{ind}_{C}(h(.,1),U).
\end{equation*}

\emph{(iv)} (Normalization) If $f$ is a constant map, with $f(u)=u_{0}$ for
every $u\in \overline{U},$ then
\begin{equation*}
\mathrm{ind}_{C}(f,U)=%
\begin{cases}
1,\quad \text{if}\ u_{0}\in U \\
0,\quad \text{if}\ u_{0}\notin \overline{U}.%
\end{cases}%
\end{equation*}

In particular, $\mathrm{ind}_{C}(f,C)=1$ for every compact function $%
f:C\rightarrow C,$ since $f$ is homotopic to any $u_{0}\in C,\ $by the
convexity of $C$ (take $h\left( u,t\right) =tf\left( u\right) +\left(
1-t\right) u_{0}$).
\end{prop}

\subsection{Solutions with at least one nonzero component}

We begin with four theorems on the existence and localization of one
solution of the operator equation $N(u,v)=(u,v)$. The first two ones assume
that the operator $N$ leaves invariant the set
\begin{equation*}
C=\left\{ \left( u,v\right) \in K:\left\vert u\right\vert \leq R_{1},\
\left\vert v\right\vert \leq R_{2}\right\} ,
\end{equation*}%
for some fixed numbers $R_{1}, R_{2}$.

\begin{thm}
\label{thm:abstr1} Assume that there exist numbers $r_{i}> 0$ and $R_{i}>0$
with
\begin{equation}
r_{i}<\Vert \phi _{i}\Vert \Vert \chi _{i}\Vert R_{i}\ \ \ (i=1,2),
\label{5}
\end{equation}%
such that
\begin{equation}
\inf\limits_{\substack{ (u,v)\in C  \\ \Vert u\Vert =r_{1},\Vert v\Vert \leq
r_{2}}}\Vert N_{1}(u,v)\Vert \geq \frac{r_{1}}{\Vert \chi _{1}\Vert },\ \
\inf\limits_{\substack{ (u,v)\in C  \\ \Vert u\Vert \leq r_{1},\Vert v\Vert
=r_{2}}}\Vert N_{2}(u,v)\Vert \geq \frac{r_{2}}{\Vert \chi _{2}\Vert },
\label{6}
\end{equation}%
and%
\begin{equation}
\sup\limits_{(u,v)\in C}|N_{i}(u,v)|\leq R_{i}\ \ \ \ (i=1,2).  \label{4}
\end{equation}%
Then $N$ has at least one fixed point $(u,v)\in K$ such that $|u|\leq R_{1},$
$|v|\leq R_{2}$ and either $\Vert u\Vert \geq r_{1}$ or $\Vert v\Vert \geq
r_{2}.$
\end{thm}

\begin{proof}
The assumption (\ref{4}) implies that $N(C)\subset C.$ Therefore, by
Proposition~\ref{fpi-pro}, $\mathrm{ind}_{C}(N,C)=1.$ Let
\begin{equation*}
U:=\left\{ \left( u,v\right) \in C:\Vert u\Vert <r_{1},\ \Vert v\Vert
<r_{2}\right\} .
\end{equation*}%
The boundary $\partial U$ of $U$ with respect to $C$ is equal to $\partial
U=A_{1}\cup A_{2},$ where
\begin{eqnarray*}
A_{1} &=&\left\{ \left( u,v\right) \in C:\left\Vert u\right\Vert =r_{1},\
\left\Vert v\right\Vert \leq r_{2}\right\} ,\  \\
A_{2} &=&\left\{ \left( u,v\right) \in C:\left\Vert u\right\Vert \leq
r_{1},\ \left\Vert v\right\Vert =r_{2}\right\} .
\end{eqnarray*}%
If there is a fixed point $(u,v)$ of $N$ on $\partial U,$ then $(u,v)$
satisfies the assertion. If not, the indices $\mathrm{ind}_{C}(N,U)$ and $%
\mathrm{ind}_{C}(N,C\setminus \overline{U})$ are well defined and their sum,
by the additivity property of the index, is equal to one. Therefore, it
suffices to prove that $\mathrm{ind}_{C}(N,U)=0.$ Take $h=(R_{1}\phi
_{1},R_{2}\phi _{2})\in C$ and consider the homotopy $H:C\times \left[ 0,1%
\right] \rightarrow C,$
\begin{equation*}
H\left( u,v,t\right) :=th+(1-t)N(u,v).
\end{equation*}%
We claim that $H$ is fixed point free on $\partial U$. Since
\begin{eqnarray}
\left\Vert R_{1}\phi _{1}\right\Vert &=&R_{1}\left\Vert \phi _{1}\right\Vert
\geq R_{1}\left\Vert \phi _{1}\right\Vert \left\Vert \chi _{1}\right\Vert
>r_{1},  \label{star} \\
\left\Vert R_{2}\phi _{2}\right\Vert &=&R_{2}\left\Vert \phi _{2}\right\Vert
\geq R_{2}\left\Vert \phi _{2}\right\Vert \left\Vert \chi _{2}\right\Vert
>r_{2},  \notag
\end{eqnarray}%
we have that $\left( u,v\right) \neq h=H\left( u,v,1\right) $ for all $%
\left( u,v\right) \in \partial U.$ It remains to show that $H\left(
u,v,t\right) \neq \left( u,v\right) $ for $\left( u,v\right) \in \partial U$
and $t\in \left( 0,1\right) .$ Assume the contrary. Then there exists $%
(u,v)\in A_{1}\cup A_{2}$ and $t\in (0,1)$ such that
\begin{equation}
(u,v)=th+(1-t)N(u,v).  \label{eq:homotopy not good}
\end{equation}%
Suppose that $(u,v)\in A_{1}.$ Then, exploiting the first coordinate of the
equation (\ref{eq:homotopy not good}), we obtain that%
\begin{eqnarray}
u &=&tR_{1}\phi _{1}+(1-t)N_{1}(u,v)\geq tR_{1}\Vert \phi _{1}\Vert \chi
_{1}+(1-t)\Vert N_{1}(u,v)\Vert \chi _{1}  \label{br0} \\
&\geq &\left( tR_{1}\Vert \phi _{1}\Vert +(1-t)\inf\limits_{\substack{ %
(u,v)\in C  \\ \Vert u\Vert =r_{1},\Vert v\Vert \leq r_{2}}}\Vert
N_{1}(u,v)\Vert \right) \chi _{1}  \notag \\
&\geq &\left( tR_{1}\Vert \phi _{1}\Vert +(1-t)\frac{r_{1}}{\Vert \chi
_{1}\Vert }\right) \chi _{1}.  \notag
\end{eqnarray}%
Using the monotonicity of $\Vert \cdot \Vert $ and (\ref{5}) we obtain that
\begin{eqnarray*}
r_{1} &\geq &\left( tR_{1}\Vert \phi _{1}\Vert +(1-t)\frac{r_{1}}{\Vert \chi
_{1}\Vert }\right) \left\Vert \chi _{1}\right\Vert \\
&=&tR_{1}\Vert \phi _{1}\Vert \left\Vert \chi _{1}\right\Vert +\left(
1-t\right) r_{1}>tr_{1}+\left( 1-t\right) r_{1}=r_{1},
\end{eqnarray*}%
which is impossible. Similarly we derive a contradiction if $\left(
u,v\right) \in A_{2}.$

By the homotopy invariance of the index we obtain that $\mathrm{ind}_C(N,U)=%
\mathrm{ind}_C(h,U).$ From (\ref{star}) we have $h\not\in \overline{U},$
hence $\mathrm{ind}_C(N,U)=\mathrm{ind}_C(h,U)=0,$ as we wished.
\end{proof}

\begin{rem}
We observe that, using the relation \eqref{eq:seminorm-cont}, a lower bound
for the solution in terms of the seminorm provides a lower bound for the
norm of the solution, namely
\begin{equation*}
\|u\|\geq r_1 \implies |u|\geq r_1|\chi_1|.
\end{equation*}
\end{rem}

In the next result we replace, in the spirit of Lemma~4 of \cite{gipp-na},
the assumption \eqref{6} with a different one. The two conditions are not
comparable and are used, in a combined way, in Theorem~\ref{thm:abstr3sols''}%
.

\begin{thm}
\label{thm:abstr1or A} Assume that there exist numbers $r_{i}> 0$ and $%
R_{i}>0$ with
\begin{equation}
r_{i}<\Vert \phi _{i}\Vert \Vert \chi _{i}\Vert R_{i}\ \ \ (i=1,2),
\end{equation}%
such that
\begin{equation}
\sup\limits_{(u,v)\in C}|N_{i}(u,v)|\leq R_{i}\ \ \ \ (i=1,2),  \label{4-4}
\end{equation}%
and
\begin{equation}
\inf_{(u,v)\in A} \|N_1(u,v)\| \geq r_1\ {\text{ or }}\ \inf_{(u,v)\in A}
\|N_2(u,v)\| \geq r_2,  \label{6 or A}
\end{equation}
where $A$ is a subset of the set
\begin{equation*}
U=\mbox{$\left\{(u,v)\in C:\;\|u\|< r_1,\ \|v\|< r_2\right\}$}.
\end{equation*}
Then $N$ has at least one fixed point $(u,v)\in K$ such that $|u|\leq R_{1},$
$|v|\leq R_{2}$ and $(u,v)\not\in A.$
\end{thm}

\begin{proof}
Since $N$ is a completely continuous mapping in the bounded closed convex
set $C,$ by Schauder's fixed point theorem, it possesses a fixed point $%
(u,v)\in C.$ We now show that the fixed point is not in $A.$ Suppose on the
contrary that $(u,v)=N(u,v)$ and $(u,v)\in A.$ Suppose that the first
inequality from (\ref{6 or A}) is satisfied. Then
\begin{equation*}
r_{1}>\Vert u\Vert =\Vert N_{1}(u,v)\Vert \geq r_{1},
\end{equation*}%
which is impossible. Similarly we arrive at a contradiction, if the second
inequality from (\ref{6 or A}) is satisfied.
\end{proof}

The next theorems do not assume the invariance condition $N\left( C\right)
\subset C$ and use instead Leray-Schauder type conditions. The first result
requires the Leray-Schauder condition componentwise.

\begin{thm}
\label{thm:retr1} Assume that there exist numbers $r_{i}>0$ and $R_{i}>0$
with
\begin{equation}
r_{i}<\Vert \phi _{i}\Vert \Vert \chi _{i}\Vert R_{i}\ \ \ (i=1,2),
\end{equation}%
such that the strengthened condition \emph{(\ref{6})}:%
\begin{equation}
\begin{split}
\inf\limits_{\substack{ (u,v)\in C  \\ \Vert u\Vert =r_{1},\Vert v\Vert \leq
r_{2}}}\left( \max \left\{ \frac{|N_{1}(u,v)|}{R_{1}},1\right\} \right)
^{-1}\Vert N_{1}(u,v)\Vert & \geq \frac{r_{1}}{\Vert \chi _{1}\Vert },
\label{6-strong} \\
\inf\limits_{\substack{ (u,v)\in C  \\ \Vert u\Vert \leq r_{1},\Vert v\Vert
=r_{2}}}\left( \max \left\{ \frac{|N_{2}(u,v)|}{R_{2}},1\right\} \right)
^{-1}\Vert N_{2}(u,v)\Vert & \geq \frac{r_{2}}{\Vert \chi _{2}\Vert },
\end{split}%
\end{equation}%
and the weakened condition \emph{(\ref{4})}:
\begin{align}
|u|& =R_{1},\ |v|<R_{2}\ \ \text{implies\ }\ N(u,v)\neq (\lambda u,v),\
\text{\ for all }\lambda >1;  \notag \\
|u|& <R_{1},\ |v|=R_{2}\ \ \text{implies}\ \ N(u,v)\neq (u,\lambda v),\
\text{for\ all\ }\lambda >1;  \label{eq:lambda} \\
|u|& =R_{1},\ |v|=R_{2}\ \ \text{implies}\ \ N(u,v)\neq (\lambda
_{1}u,\lambda _{2}v),\ \text{for all }\lambda _{1},\lambda _{2}\geq 1\text{
with}\ \lambda _{1}\lambda _{2}>1,  \notag
\end{align}%
hold. Then $N$ has at least one fixed point $(u,v)\in K$ such that $|u|\leq
R_{1},$ $|v|\leq R_{2}$ and either $\Vert u\Vert \geq r_{1}$ or $\Vert
v\Vert \geq r_{2}.$
\end{thm}

\begin{proof}
Consider the retraction $\pi \colon K\rightarrow C,$ $\pi (u,v)=(\pi
_{1}(u),\pi _{2}(v)),$ where
\begin{equation*}
\pi _{i}\colon K_{i}\rightarrow C_{i}:=\left\{ u\in K_{i}:\left\vert
u\right\vert \leq R_{i}\right\} ,\ \ \pi _{i}(u)=%
\begin{cases}
u & \text{if\ }\left\vert u\right\vert \leq R_{i} \\
\frac{R_{i}}{|u|}u & \text{if\ }\left\vert u\right\vert \geq R_{i}.%
\end{cases}%
\end{equation*}%
Now we define the operator $\tilde{N}$ by the formula
\begin{equation*}
\tilde{N}(u,v)=\pi (N(u,v))=(\pi _{1}(N_{1}(u,v)),\pi _{2}(N_{2}(u,v))).
\end{equation*}%
Then $\tilde{N}(C)\subset C,$ i.e. $\tilde{N}$ satisfies (\ref{4}). Now, let
$(u,v)\in C$ be such that $\Vert u\Vert =r_{1},\ \Vert v\Vert \leq r_{2}.$
Observe that
\begin{equation*}
\tilde{N}_{1}(u,v)=\pi _{1}(N_{1}(u,v))=\left( \max \left\{ \frac{%
|N_{1}(u,v)|}{R_{1}},1\right\} \right) ^{-1}N_{1}(u,v),
\end{equation*}%
which in view of (\ref{6-strong}) shows that
\begin{equation*}
\Vert \tilde{N}_{1}(u,v)\Vert \geq \frac{r_{1}}{\Vert \chi _{1}\Vert }.
\end{equation*}%
A similar estimate holds for $\tilde{N}_{2},$ which shows that $\tilde{N}$
satisfies (\ref{6}). By Theorem \ref{thm:abstr1} we obtain a fixed point $%
(u,v)$ of $\tilde{N}$ in the set $C\setminus U.$ Therefore $\pi
(N(u,v))=(u,v).$ If $\left\vert N_{1}\left( u,v\right) \right\vert \leq
R_{1},$ then $\pi _{1}\left( N_{1}\left( u,v\right) \right) =N_{1}\left(
u,v\right) $ and so $N_{1}\left( u,v\right) =u.$ Similarly, if $\left\vert
N_{2}\left( u,v\right) \right\vert \leq R_{2},$ then $N_{2}\left( u,v\right)
=v,$ and the proof is complete. Otherwise, we have $\left\vert N_{1}\left(
u,v\right) \right\vert >R_{1}$ or $\left\vert N_{2}\left( u,v\right)
\right\vert >R_{2}.$ Assume $\left\vert N_{1}\left( u,v\right) \right\vert
>R_{1}.$ Then
\begin{equation*}
\pi _{1}\left( N_{1}\left( u,v\right) \right) =\frac{R_{1}}{\left\vert
N_{1}\left( u,v\right) \right\vert }N_{1}\left( u,v\right) =u.
\end{equation*}%
It follows that $N_{1}\left( u,v\right) =\frac{\left\vert N_{1}\left(
u,v\right) \right\vert }{R_{1}}u$ and $|u|=R_{1}$ which, in view of the
first implication from (\ref{eq:lambda}), is impossible if $\left\vert
v\right\vert <R_{2}.$ Hence $\left\vert v\right\vert =R_{2},$ which implies
that $N_{2}\left( u,v\right) =\lambda _{2}v$ for some $\lambda _{2}\geq 1.$
Since $N_{1}\left( u,v\right) =\lambda _{1}u,$ where $\lambda _{1}=$ $\frac{%
\left\vert N_{1}\left( u,v\right) \right\vert }{R_{1}}>1$ and $\left\vert
u\right\vert =R_{1},$ we contradict the third condition from (\ref{eq:lambda}%
). Thus the inequality $\left\vert N_{1}\left( u,v\right) \right\vert >R_{1}$
is not possible. By symmetry, the inequality $\left\vert N_{2}\left(
u,v\right) \right\vert >R_{2}$ is also not possible. Consequently, $\pi
(N(u,v))=N\left( u,v\right) =\left( u,v\right) .$
\end{proof}

\begin{rem}
Notice that under condition (\ref{4}), (\ref{eq:lambda}) is satisfied and (%
\ref{6-strong}) reduces to (\ref{6}).
\end{rem}

If instead of the retraction $\pi ,$ we consider the retraction $\rho
:K\rightarrow C$ given by
\begin{equation*}
\rho (u,v)=\left( \max \left\{ \frac{\left\vert u\right\vert }{R_{1}},\frac{%
\left\vert v\right\vert }{R_{2}},1\right\} \right) ^{-1}(u,v),
\end{equation*}%
we obtain the following existence theorem under the Leray-Schauder condition
acting this time, uniformly on the two components $u,v.$

\begin{thm}
\label{thm:retr2} Assume that there exist numbers $r_{i}>0$ and $R_{i}>0$
with
\begin{equation}
r_{i}<\Vert \phi _{i}\Vert \Vert \chi _{i}\Vert R_{i}\ \ \ (i=1,2),
\end{equation}%
such that the strengthened condition \emph{(\ref{6})}:
\begin{equation}
\begin{split}
\inf\limits_{\substack{ (u,v)\in C  \\ \Vert u\Vert =r_{1},\Vert v\Vert \leq
r_{2}}}\left( \max \left\{ \frac{|N_{1}(u,v)|}{R_{1}},\frac{|N_{2}(u,v)|}{%
R_{2}},1\right\} \right) ^{-1}\Vert N_{1}(u,v)\Vert & \geq \frac{r_{1}}{%
\Vert \chi _{1}\Vert },  \label{6-stronger} \\
\inf\limits_{\substack{ (u,v)\in C  \\ \Vert u\Vert \leq r_{1},\Vert v\Vert
=r_{2}}}\left( \max \left\{ \frac{|N_{1}(u,v)|}{R_{1}},\frac{|N_{2}(u,v)|}{%
R_{2}},1\right\} \right) ^{-1}\Vert N_{2}(u,v)\Vert & \geq \frac{r_{2}}{%
\Vert \chi _{2}\Vert },
\end{split}%
\end{equation}%
and the weakened condition \emph{(\ref{4})}:%
\begin{equation}
N(u,v)\neq \lambda (u,v),\ \ \ \text{for \ }(u,v)\in \partial C,\ \lambda >1,
\label{eq:lambda2}
\end{equation}%
hold. Then $N$ has at least one fixed point $(u,v)\in K$ such that $|u|\leq
R_{1},$ $|v|\leq R_{2}$ and either $\Vert u\Vert \geq r_{1}$ or $\Vert
v\Vert \geq r_{2}.$
\end{thm}

\begin{proof}
The assumption (\ref{6-stronger}) guarantees that the mapping $\bar{N}%
(u,v)=\rho (N(u,v)),$ that has the property $\bar{N}(C)\subset C,$ also
satisfies condition (\ref{6}). By Theorem \ref{thm:abstr1}, $\bar{N}$ has a
fixed point $\left( u,v\right) $ in $C\setminus U.$ If $N\left( u,v\right)
\in C,$ then $\rho \left( N\left( u,v\right) \right) =N\left( u,v\right) $
and we are done. Otherwise, $\left\vert N_{1}\left( u,v\right) \right\vert
>R_{1}$ or $\left\vert N_{2}\left( u,v\right) \right\vert >R_{2},$ and then $%
\lambda :=\max \left\{ \frac{\left\vert N_{1}\left( u,v\right) \right\vert }{%
R_{1}},\frac{\left\vert N_{2}\left( u,v\right) \right\vert }{R_{2}}%
,1\right\} >1,$ $N\left( u,v\right) =\lambda \left( u,v\right) $ and $\left(
u,v\right) \in \partial C,$ which is excluded by (\ref{eq:lambda2}).
\end{proof}

\begin{rem}
Note that, in the case of one equation, the retractions $\pi$ and $\rho$
coincide with the one used in \cite[Theorem 2.1]{b:RP_TMNA}, but are used
differently.
\end{rem}

\subsection{Solutions with both nonzero components}

In the previous results one of the components of the solution can be $0.$
The next three theorems avoid this situation.

\begin{thm}
\label{thm:abstr1'} Assume that there exist numbers $r_{i}> 0$ and $R_{i}>0$
with
\begin{equation}
r_{i}<\Vert \phi _{i}\Vert \Vert \chi _{i}\Vert R_{i}\ \ \ (i=1,2),
\end{equation}%
such that
\begin{equation}
\inf\limits_{\substack{ (u,v)\in C  \\ \Vert u\Vert =r_{1},\text{\ }%
\left\Vert v\right\Vert \geq r_{2}}}\Vert N_{1}(u,v)\Vert \geq \frac{r_{1}}{%
\Vert \chi _{1}\Vert },\ \ \ \inf\limits_{\substack{ (u,v)\in C  \\ %
\left\Vert u\right\Vert \geq r_{1},\ \Vert v\Vert =r_{2}}}\Vert
N_{2}(u,v)\Vert \geq \frac{r_{2}}{\Vert \chi _{2}\Vert },  \label{6'}
\end{equation}%
and%
\begin{equation}
\sup\limits_{(u,v)\in C}|N_{i}(u,v)|\leq R_{i}\ \ \ \ (i=1,2).  \label{4bis}
\end{equation}
Then $N$ has at least one fixed point $(u,v)\in K$ such that $|u|\leq R_{1},$
$|v|\leq R_{2},$ $\Vert u\Vert \geq r_{1}$ {and} $\Vert v\Vert \geq r_{2}.$
\end{thm}

\begin{proof}
As before, the assumption \eqref{4bis} implies that $N(C)\subset C.$ Thus, $%
\mathrm{ind}_C(N,C)=1.$ In order to finish the proof, it is sufficient to
show that $\mathrm{ind}_C(N,V)=0,$ where
\begin{equation*}
V:=\left\{ \left( u,v\right) \in C:\left\Vert u\right\Vert <r_{1}\ \text{or\
}\left\Vert v\right\Vert <r_{2}\right\} .
\end{equation*}%
We have $\partial V=B_{1}\cup B_{2},$ where
\begin{eqnarray*}
B_{1} &=&\left\{ \left( u,v\right) \in C:\left\Vert u\right\Vert =r_{1},\
\left\Vert v\right\Vert \geq r_{2}\right\} , \\
B_{2} &=&\left\{ \left( u,v\right) \in C:\left\Vert u\right\Vert \geq
r_{1},\ \left\Vert v\right\Vert =r_{2}\right\} .
\end{eqnarray*}%
If $N$ has a fixed point $\left( u,v\right) \in \partial V,$ we are
finished. Thus, assume that $N\left( u,v\right) \neq \left( u,v\right) $ for
$\left( u,v\right) \in \partial V,$ and consider the same homotopy as in the
proof of Theorem \ref{thm:abstr1}, that is
\begin{equation*}
H\left( u,v,t\right) =th+(1-t)N(u,v),\ \text{ where }\ h=(R_{1}\phi
_{1},R_{2}\phi _{2}).
\end{equation*}%
We claim that $H$ is fixed point free on $\partial V.$ From the previous
assumption, $H\left( u,v,0\right) \neq \left( u,v\right) $ for $\left(
u,v\right) \in \partial V.$ Also, as in the proof of Theorem \ref{thm:abstr1}%
, $H\left( u,v,1\right) =h\neq \left( u,v\right) $ on $\partial V.$ It
remains to show that $H\left( u,v,t\right) \neq \left( u,v\right) $ for $%
\left( u,v\right) \in \partial V$ and $t\in \left( 0,1\right) .$ Assume the
contrary. Then, there exists $(u,v)\in B_{1}\cup B_{2}$ and $t\in (0,1)$
such that
\begin{equation}
(u,v)=th+(1-t)N(u,v).  \label{eq:homotopy not good'}
\end{equation}%
Suppose that $(u,v)\in B_{1}.$ Then, exploiting the first coordinate of the
equation (\ref{eq:homotopy not good'}), we obtain that
\begin{align}
u& =tR_{1}\phi _{1}+(1-t)N_{1}(u,v)\geq tR_{1}\Vert \phi _{1}\Vert \chi
_{1}+(1-t)\Vert N_{1}(u,v)\Vert \chi _{1}  \notag \\
& \geq \left( tR_{1}\Vert \phi _{1}\Vert +(1-t)\inf\limits_{\substack{ %
(u,v)\in C  \\ \Vert u\Vert =r_{1},\Vert v\Vert \geq r_{2}}}\Vert
N_{1}(u,v)\Vert \right) \chi _{1}\geq \left( tR_{1}\Vert \phi _{1}\Vert
+(1-t)\frac{r_{1}}{\Vert \chi _{1}\Vert }\right) \chi _{1}
\end{align}%
thus, as in the proof of Theorem \ref{thm:abstr1}, we obtain the
contradiction $r_{1}>r_{1}.$ The case $\left( u,v\right) \in B_{2}$ is
similar. Therefore $\mathrm{ind}_C(N,V)=\mathrm{ind}_C(h,V)=0$ since in view
of (\ref{star}), $h\notin \overline{V}.$
\end{proof}

Theorem \ref{thm:abstr1'} can be generalized in the spirit of Theorems \ref%
{thm:retr1} and \ref{thm:retr2}.

\begin{thm}
\label{thm:retr1'} Assume that there exist numbers $r_{i}>0$ and $R_{i}>0$
with
\begin{equation}
r_{i}<\Vert \phi _{i}\Vert \Vert \chi _{i}\Vert R_{i}\ \ \ (i=1,2),
\end{equation}%
such that
\begin{equation}
\begin{split}
\inf\limits_{\substack{ (u,v)\in C  \\ \Vert u\Vert =r_{1},\Vert v\Vert \geq
r_{2}}}\left( \max \left\{ \frac{|N_{1}(u,v)|}{R_{1}},1\right\} \right)
^{-1}\Vert N_{1}(u,v)\Vert & \geq \frac{r_{1}}{\Vert \chi _{1}\Vert }, \\
\inf\limits_{\substack{ (u,v)\in C  \\ \Vert u\Vert \geq r_{1},\Vert v\Vert
=r_{2}}}\left( \max \left\{ \frac{|N_{2}(u,v)|}{R_{2}},1\right\} \right)
^{-1}\Vert N_{2}(u,v)\Vert & \geq \frac{r_{2}}{\Vert \chi _{2}\Vert },
\end{split}
\label{6-strong'}
\end{equation}%
and
\begin{align}
|u|& =R_{1},\ |v|<R_{2}\ \ \text{implies\ }\ N(u,v)\neq (\lambda u,v),\
\text{\ for all }\lambda >1;  \notag \\
|u|& <R_{1},\ |v|=R_{2}\ \ \text{implies}\ \ N(u,v)\neq (u,\lambda v),\
\text{for\ all\ }\lambda >1;  \label{eq:lambda-bis} \\
|u|& =R_{1},\ |v|=R_{2}\ \ \text{implies}\ \ N(u,v)\neq (\lambda
_{1}u,\lambda _{2}v),\ \text{for all }\lambda _{1},\lambda _{2}\geq 1\text{
with}\ \lambda _{1}\lambda _{2}>1.  \notag
\end{align}%
Then $N$ has at least one fixed point $(u,v)\in K$ such that $|u|\leq R_{1},$
$|v|\leq R_{2},$ $\Vert u\Vert \geq r_{1}$ {and} $\Vert v\Vert \geq r_{2}.$
\end{thm}

\begin{proof}
Define the retractions $\pi _{i}$ and the operator $\tilde{N}$ as in the
proof of Theorem \ref{thm:retr1}. Then $\tilde{N}(C)\subset C,$ i.e. $\tilde{%
N}$ satisfies (\ref{4bis}). Now, let $(u,v)\in C$ be such that $\Vert u\Vert
=r_{1},$ $\Vert v\Vert \geq r_{2}.$ Observe that
\begin{equation*}
\tilde{N}_{1}(u,v)=\pi _{1}(N_{1}(u,v))=\left( \max \left\{ \frac{%
|N_{1}(u,v)|}{R_{1}},1\right\} \right) ^{-1}N_{1}(u,v),
\end{equation*}%
which shows that
\begin{equation*}
\Vert \tilde{N}_{1}(u,v)\Vert \geq \frac{r_{1}}{\Vert \chi _{1}\Vert }.
\end{equation*}%
A similar estimate holds for $\tilde{N}_{2}$ which shows that $\tilde{N}$
satisfies (\ref{6'}). By Theorem \ref{thm:abstr1'} we obtain a fixed point $%
(u,v)$ of $\tilde{N}$ in the set $C\setminus V.$ Therefore $\pi
(N(u,v))=(u,v).$ In the proof of Theorem \ref{thm:retr1} it was shown that (%
\ref{eq:lambda-bis}) yields $N(u,v)=(u,v).$
\end{proof}

Similarly, using the method presented in the proof of Theorem \ref{thm:retr2}
and exploiting Theorem \ref{thm:abstr1'}, we obtain the following fact.

\begin{thm}
\label{thm:retr2'}

Assume that there exist numbers $r_{i}>0$ and $R_{i}>0$ with
\begin{equation}
r_{i}<\Vert \phi _{i}\Vert \Vert \chi _{i}\Vert R_{i}\ \ \ (i=1,2),
\end{equation}%
such that
\begin{equation}
\begin{split}
\inf\limits_{\substack{ (u,v)\in C  \\ \Vert u\Vert =r_{1},\Vert v\Vert \geq
r_{2}}}\left( \max \left\{ \frac{|N_{1}(u,v)|}{R_{1}},\frac{|N_{2}(u,v)|}{%
R_{2}},1\right\} \right) ^{-1}\Vert N_{1}(u,v)\Vert & \geq \frac{r_{1}}{%
\Vert \chi _{1}\Vert }, \\
\inf\limits_{\substack{ (u,v)\in C  \\ \Vert u\Vert \geq r_{1},\Vert v\Vert
=r_{2}}}\left( \max \left\{ \frac{|N_{1}(u,v)|}{R_{1}},\frac{|N_{2}(u,v)|}{%
R_{2}},1\right\} \right) ^{-1}\Vert N_{2}(u,v)\Vert & \geq \frac{r_{2}}{%
\Vert \chi _{2}\Vert },
\end{split}
\label{6-stronger'}
\end{equation}%
and
\begin{equation}
N(u,v)\neq \lambda (u,v),\ \ \ \text{for \ }(u,v)\in \partial C,\ \lambda >1.
\label{eq:lambda2-bis}
\end{equation}%
Then $N$ has at least one fixed point $(u,v)\in K$ such that $|u|\leq R_{1},$
$|v|\leq R_{2},$ $\Vert u\Vert \geq r_{1}$ {and} $\Vert v\Vert \geq r_{2}.$
\end{thm}

\begin{proof}
Let the mapping $\bar{N}$ be defined as in the proof of Theorem \ref%
{thm:retr2}. The assumption (\ref{6-stronger'}) guarantees that $\bar{N},$
having the property $\bar{N}(C)\subset C,$ also satisfies condition (\ref{6'}%
). By Theorem \ref{thm:abstr1'}, $\bar{N}$ has a fixed point $\left(
u,v\right) $ in $C\setminus V.$ It was shown in the proof of Theorem \ref%
{thm:retr2} that (\ref{eq:lambda2-bis}) implies $N(u,v)=(u,v).$
\end{proof}

\subsection{Multiplicity results}

Recall that the seminorms $\Vert \cdot \Vert _{i}$ are continuous in $K_{i}$ with
respect to the topology induced by $|\cdot|_{i},$ which implies that there exist
constants $c_{i}>0$ such that $\Vert u\Vert _{i}\leq c_{i}|u|_{i}$ for all $%
u\in K_{i}.$

\begin{thm}
\label{thm:abstr3sols} Assume that there exist numbers $\rho _{i}, r_{i},
R_{i}$ with
\begin{equation}
0<c_{i}\rho _{i}<r_{i}<\Vert \phi _{i}\Vert \Vert \chi _{i}\Vert R_{i}\ \ \
(i=1,2),
\end{equation}%
such that
\begin{equation}
\inf\limits_{\substack{ (u,v)\in C  \\ \Vert u\Vert =r_{1},\text{\ }%
\left\Vert v\right\Vert \geq r_{2}}}\Vert N_{1}(u,v)\Vert > \frac{r_{1}}{%
\Vert \chi _{1}\Vert },\ \ \ \inf\limits_{\substack{ (u,v)\in C  \\ %
\left\Vert u\right\Vert \geq r_{1},\ \Vert v\Vert =r_{2}}}\Vert
N_{2}(u,v)\Vert >\frac{r_{2}}{\Vert \chi _{2}\Vert },  \label{6'>}
\end{equation}%
\begin{equation}
\sup\limits_{(u,v)\in C}|N_{i}(u,v)|\leq R_{i}\ \ \ \ (i=1,2),  \label{4tris}
\end{equation}%
and
\begin{equation}
N(u,v)\neq \lambda (u,v),\ \text{ for }\lambda \geq 1\text{ and }(\left\vert
u\right\vert =\rho _{1},\ \left\vert v\right\vert \leq \rho _{2}\text{ or }%
\left\vert u\right\vert \leq \rho _{1},\ \left\vert v\right\vert =\rho _{2}).
\label{eq:deg1-rho}
\end{equation}%
Then $N$ has at least three fixed points $(u_{i},v_{i})\in C$ $\left(
i=1,2,3\right) $ with%
\begin{eqnarray*}
\left\vert u_{1}\right\vert &<&\rho _{1},\ \ \left\vert v_{1}\right\vert
<\rho _{2}\ \ \left( \text{possibly zero solution}\right); \\
\left\Vert u_{2}\right\Vert &<&r_{1}\ \text{or\ }\left\Vert v_{2}\right\Vert
<r_{2};\ \left\vert u_{2}\right\vert >\rho _{1}\ \text{or }\left\vert
v_{2}\right\vert >\rho _{2}\ \ \left( \text{possibly one solution component
zero}\right) ; \\
\left\Vert u_{3}\right\Vert &>&r_{1},\ \left\Vert v_{3}\right\Vert >r_{2}\ \
\left( \text{both solution components nonzero}\right).
\end{eqnarray*}
\end{thm}

\begin{proof}
Let $V$ be as in the proof of Theorem \ref{thm:abstr1'}. Strict inequalities
in (\ref{6'>}) guarantee that $N$ is fixed point free on $\partial V.$
According to the proof of Theorem \ref{thm:abstr1'} we have $\mathrm{ind}%
_{C}(N,C)=1,$ $\mathrm{ind}_{C}(N,V)=0$ and therefore by the additivity
property, $\mathrm{ind}_{C}(N,C\setminus \overline{V})=1.$ Let%
\begin{equation*}
W:=\left\{ \left( u,v\right) \in C:\left\vert u\right\vert <\rho _{1},\
\left\vert v\right\vert <\rho _{2}\right\} .\text{ }
\end{equation*}%
For every $\left( u,v\right) \in \overline{W},$ we have%
\begin{equation*}
\left\Vert u\right\Vert \leq c_{1}\left\vert u\right\vert \leq c_{1}\rho
_{1}<c_{1}\left( \frac{1}{c_{1}}r_{1}\right) =r_{1}
\end{equation*}%
and, similarly, $\left\Vert v\right\Vert <r_{2}.$ Hence $\left( u,v\right)
\in V,$ which proves that $\overline{W}\subset V.$ Condition (\ref%
{eq:deg1-rho}) shows that $N$ is homotopic with zero on $W.$ Thus $\mathrm{%
ind}_{C}(N,W)=\mathrm{ind}_{C}(0,W)=1.$ Then $\mathrm{ind}_{C}(N,V\setminus
\overline{W})=0-1=-1.$ Consequently, there exist at least three fixed points
of $N,$ in $W,\ V\setminus \overline{W}$ and $C\setminus \overline{V}.$
\end{proof}

If we assume the following estimates of the $\Vert N_{i}(u,v)\Vert :$

\begin{equation}
\inf\limits_{\substack{ (u,v)\in C  \\ \Vert u\Vert =r_{1}}}\Vert
N_{1}(u,v)\Vert >\frac{r_{1}}{\Vert \chi _{1}\Vert },\ \ \inf\limits
_{\substack{ (u,v)\in C  \\ \Vert v\Vert =r_{2}}}\Vert N_{2}(u,v)\Vert >%
\frac{r_{2}}{\Vert \chi _{2}\Vert },  \label{6''}
\end{equation}%
then we can obtain a more precise localization for the solution $\left(
u_{2},v_{2}\right) $ in Theorem \ref{thm:abstr3sols}, the Figure~\ref%
{fig:3sols} illustrates this fact.

\begin{figure}[h]
\centering
\includegraphics[width=\textwidth]{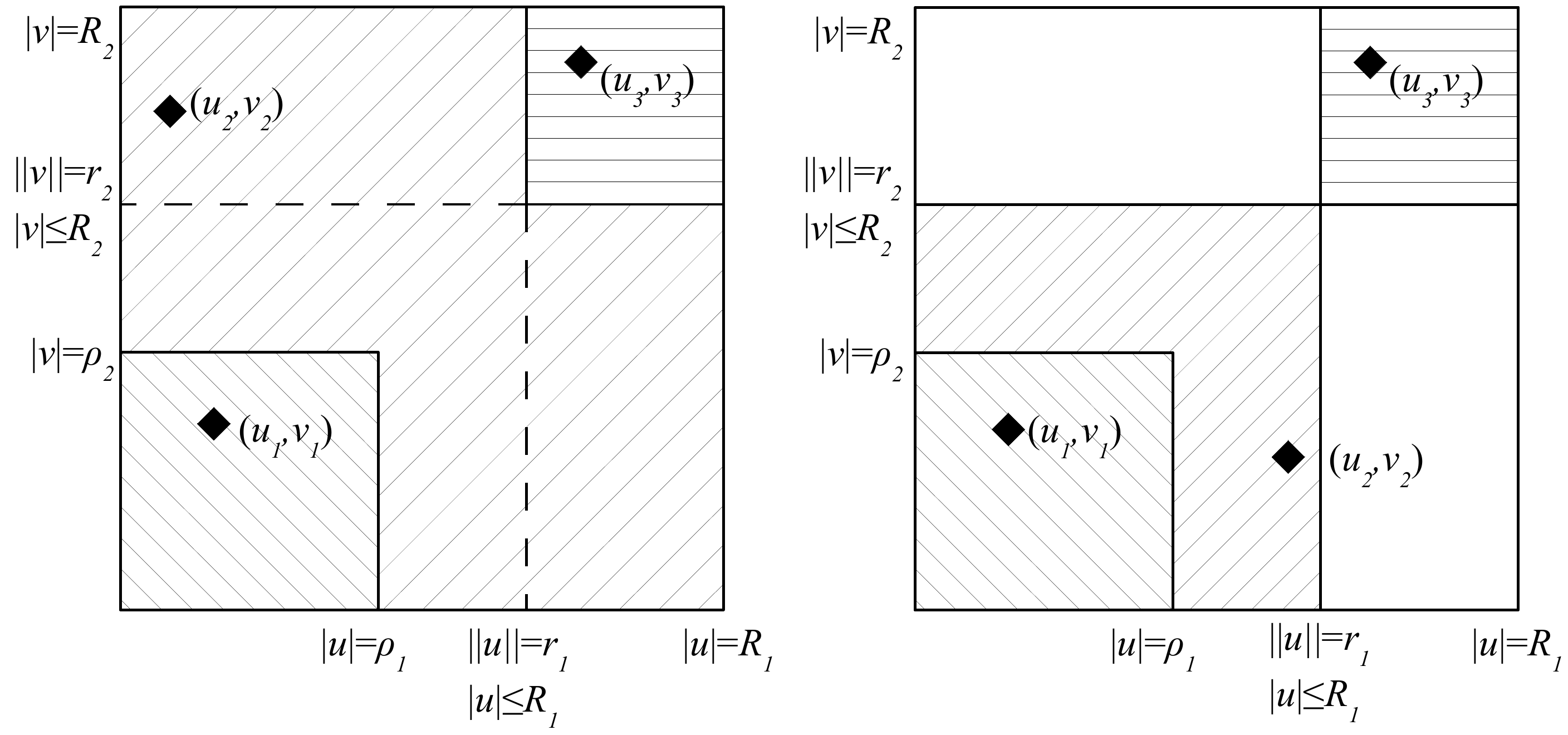}
\caption{Localization of the three solutions $(u_{i},v_{i})$ from Theorem
\protect\ref{thm:abstr3sols} (on the left) and \protect\ref{thm:abstr3sols'}
(on the right).}
\label{fig:3sols}
\end{figure}

\begin{thm}
\label{thm:abstr3sols'} Suppose that all the assumptions of Theorem \emph{%
\ref{thm:abstr3sols}} are satisfied with the condition \emph{(\ref{6'>})}
replaced by \emph{(\ref{6''})}. Then $N$ has at least three fixed points $%
(u_{i},v_{i})\in C$ $\left( i=1,2,3\right) $ with%
\begin{eqnarray*}
\left\vert u_{1}\right\vert &<&\rho _{1},\ \ \left\vert v_{1}\right\vert
<\rho _{2}\ \ \left( \text{possibly zero solution}\right); \\
\left\Vert u_{2}\right\Vert &<&r_{1},\ \Vert v_{2}\Vert <r_{2};\ \left\vert
u_{2}\right\vert >\rho _{1}\text{ or }\left\vert v_{2}\right\vert >\rho
_{2}\ \ \left( \text{possibly one solution component zero}\right) ; \\
\left\Vert u_{3}\right\Vert &>&r_{1},\ \left\Vert v_{3}\right\Vert >r_{2}\ \
\left( \text{both solution components nonzero}\right).
\end{eqnarray*}
\end{thm}

\begin{proof}
The assumption (\ref{6''}) implies both (\ref{6}) and (\ref{6'}) and that
there are no fixed points of $N$ on $\partial U$ and $\partial V.$ Hence, as
in the proofs of Theorems \ref{thm:abstr1} and \ref{thm:abstr1'}, the
indices $\mathrm{ind}_{C}(N,U)$ and $\mathrm{ind}_{C}(N,V)$ are well defined
and equal $0.$ An analysis similar to that in the proof of Theorem \ref%
{thm:abstr3sols} shows that
\begin{equation*}
\mathrm{ind}_{C}(N,W)=1,\ \mathrm{ind}_{C}(N,U\setminus \overline{W})=-1,\
\mathrm{ind}_{C}(N,C\setminus \overline{V})=1,
\end{equation*}%
which completes the proof.
\end{proof}

In order to ensure that the solution $(u_{1},v_{1})$ from the theorems above
is nonzero, and thereby to obtain three \emph{nonzero} solutions, we use
some additional assumptions on $N.$

\begin{thm}
\label{thm:abstr3sols''}Assume that all the conditions of Theorem \emph{\ref%
{thm:abstr3sols}} or Theorem \emph{\ref{thm:abstr3sols'}} are satisfied.
Consider $0<\varrho _{i}<\Vert \phi _{i}\Vert \Vert \chi _{i}\Vert \rho _{i},
$ $i=1,2.$

\emph{(i)} If
\begin{equation}
\inf\limits_{\substack{ (u,v)\in K,|u|\leq \rho _{1},|v|\leq \rho _{2} \\ %
\Vert u\Vert =\varrho _{1},\Vert v\Vert \geq \varrho _{2}}}\Vert
N_{1}(u,v)\Vert \geq \frac{\varrho _{1}}{\Vert \chi _{1}\Vert },\ \
\inf\limits_{\substack{ (u,v)\in K,|u|\leq \rho _{1},|v|\leq \rho _{2} \\ %
\Vert u\Vert \geq \varrho _{1},\Vert v\Vert =\varrho _{2}}}\Vert
N_{2}(u,v)\Vert \geq \frac{\varrho _{2}}{\Vert \chi _{2}\Vert }
\label{6varrho}
\end{equation}%
and
\begin{equation*}
|N_{i}(u,v)|\leq \rho _{i}\ \ \text{for\ \ }|u|\leq \rho _{1},\ |v|\leq \rho
_{2}\ \ \ \left( i=1,2\right) ,
\end{equation*}%
then we can assume that the solution $(u_{1},v_{1})$ from Theorem \emph{\ref%
{thm:abstr3sols}} or \emph{\ref{thm:abstr3sols'}} satisfies $\Vert
u_{1}\Vert \geq \varrho _{1}$ and $\Vert v_{1}\Vert \geq \varrho _{2};$

\emph{(ii)} if
\begin{equation}
\inf\limits_{\substack{ (u,v)\in K,|u|\leq \tilde{\rho}_{1},|v|\leq \tilde{%
\rho}_{2}  \\ \Vert u\Vert \leq \varrho _{1},\Vert v\Vert \leq \varrho _{2}}}%
\Vert N_{1}(u,v)\Vert \geq \varrho _{1}\ {\text{ or }}\ \inf\limits
_{\substack{ (u,v)\in K,|u|\leq \tilde{\rho}_{1},|v|\leq \tilde{\rho}_{2}
\\ \Vert u\Vert \leq \varrho _{1},\Vert v\Vert \leq \varrho _{2}}}\Vert
N_{2}(u,v)\Vert \geq \varrho _{2}  \label{6 or varrho}
\end{equation}%
for some$\ \tilde{\rho}_{1}\leq \rho _{1},\ \tilde{\rho}_{2}\leq \rho _{2},$
then we can assume that the solution $(u_{1},v_{1})$ from Theorem \emph{\ref%
{thm:abstr3sols}} or \emph{\ref{thm:abstr3sols'}} satisfies $\Vert
u_{1}\Vert \geq \varrho _{1}$ or $\Vert v_{1}\Vert \geq \varrho _{2}$ or $%
|u_{1}|>\tilde{\rho}_{1}$ or $|v_{1}|>\tilde{\rho}_{2}.$
\end{thm}

\begin{proof}
(i) The inequality follows from Theorem \ref{thm:abstr1'} applied in the
case of $r_{i}:=\varrho _{i}$ and $R_{i}:=\rho _{i}.$

(ii) From Theorem \ref{thm:abstr1or A} applied in the case of $%
r_{i}:=\varrho _{i},$ $R_{i}:=\rho _{i}$ and
\begin{equation*}
A=%
\mbox{$\left\{(u,v):\;\|u\| < \varrho_1, \|v\| < \varrho_2, |u|\leq
\tilde\rho_1, |v|\leq \tilde\rho_2\right\}$},
\end{equation*}%
we obtain that there are no fixed points in $N$ in $A,$ which ends the proof.
\end{proof}

The next Remark illustrates how Theorem \ref{thm:abstr1'} can be used to
prove the existence of $n$ nontrivial solutions.

\begin{rem}
If $N$ satisfies the conditions of Theorem \ref{thm:abstr1'} for all pairs
of radii
\begin{equation*}
r_{i}^{j}\leq \Vert \phi _{i}\Vert \Vert \chi _{i}\Vert R_{i}^{j}\ \text{
for }\ i=1,2,\ j=1,2,\ldots ,n,
\end{equation*}%
satisfying
\begin{equation*}
c_{i}R_{i}^{j}<r_{i}^{j+1}\ \text{ for }\ i=1,2,\ j=1,2,\ldots ,n-1,
\end{equation*}%
then $N$ possesses at least $n$ nontrivial solutions $(u_{j},v_{j})$ with
\begin{equation*}
|u_{j}|\leq R_{1}^{j},\ |v_{j}|\leq R_{2}^{j},\ \Vert u_{j}\Vert \geq
r_{1}^{j},\ \Vert v_{j}\Vert \geq r_{2}^{j}.
\end{equation*}%
Moreover, if the conditions
\begin{equation*}
\inf\limits_{\substack{ (u,v)\in K,|u|\leq r_{1}^{j},|v|\leq r_{2}^{j}  \\ %
\Vert u\Vert =r_{1}^{j},\text{\ }\left\Vert v\right\Vert \geq r_{2}^{j}}}%
\Vert N_{1}(u,v)\Vert >\frac{r_{1}^{j}}{\Vert \chi _{1}\Vert },\ \ \
\inf\limits_{\substack{ (u,v)\in K,|u|\leq r_{1}^{j},|v|\leq r_{2}^{j}  \\ %
\left\Vert u\right\Vert \geq r_{1}^{j},\ \Vert v\Vert =r_{2}^{j}}}\Vert
N_{2}(u,v)\Vert >\frac{r_{2}^{j}}{\Vert \chi _{2}\Vert },
\end{equation*}%
\begin{equation*}
\sup\limits_{(u,v)\in K,|u|\leq R_{1}^{j},|v|\leq
R_{2}^{j}}|N_{i}(u,v)|<R_{i}^{j}\ \ \ \ (i=1,2),
\end{equation*}%
hold, then we have $n-1$ additional solutions $(\bar{u}_{j},\bar{v}_{j}),$ $%
j=1,\ldots ,n-1$ such that
\begin{equation*}
|\bar{u}_{j}|<R_{1}^{j+1},\ |\bar{v}_{j}|<R_{2}^{j+1};\ |\bar{u}%
_{j}|>R_{1}^{j}\text{ or }|\bar{v}_{j}|>R_{2}^{j};\ \Vert \bar{u}_{j}\Vert
<r_{1}^{j+1}\text{ or }\Vert \bar{v}_{j}\Vert <r_{2}^{j+1}.
\end{equation*}%
The first conclusion follows from Theorem \ref{thm:abstr1'} applied $n$
times, whereas the second follows from Theorem \ref{thm:abstr3sols} applied $%
n-1$ times. 
\end{rem}

\begin{rem}
\label{rem:betterconstant} Assume that between the cones $K_{i}\subseteq
K_{i}^{0}$ and the norms $\left\vert \cdot\right\vert _{i}$ $\left( i=1,2\right)
$ there is a strong compatibility, expressed by the following condition:%
\begin{equation}
\text{there exists }h_{i}^{0}\in K_{i},\,\left\vert h_{i}^{0}\right\vert =1,%
\text{ such that }h_{i}^{0}\geq u\text{ for every }u\in K_{i}\text{ with }%
\left\vert u\right\vert \leq 1.  \label{br}
\end{equation}%
Then the localization theory for fixed points of a given operator $N=\left(
N_{1},N_{2}\right) :K\rightarrow K$ can be developed without assuming the
Harnack type inequality $``u\geq \left\Vert u\right\Vert \chi _{i}"$ for the
definition of $K$ and without specifying the elements $\chi _{i},\phi _{i}.$
This is based on a different way of proving that the homotopy $%
H(u,v,t)=th+(1-t)N\left( u,v\right) $ is fixed point free on the boundary.
Here we choose $\ h=\left( R_{1}h_{1}^{0},R_{2}h_{2}^{0}\right) .$ For
instance, the corresponding version of Theorem \ref{thm:abstr1} is the
following:
\end{rem}

\begin{thm}
Assume that condition \emph{\eqref{br}} holds, that $N:K\rightarrow K$ is
completely continuous and that for some numbers $0\leq r_{i}<\left\Vert
h_{i}^{0}\right\Vert R_{i},$ the conditions
\begin{equation}
\inf\limits_{\substack{ (u,v)\in C  \\ \Vert u\Vert =r_{1},\Vert v\Vert \leq
r_{2}}}\Vert N_{1}(u,v)\Vert >r_{1},\ \ \inf\limits_{\substack{ (u,v)\in C
\\ \Vert u\Vert \leq r_{1},\Vert v\Vert =r_{2}}}\Vert N_{2}(u,v)\Vert >r_{2}
\label{br1}
\end{equation}%
and%
\begin{equation*}
\sup\limits_{(u,v)\in C}|N_{i}(u,v)|\leq R_{i}\ \ \ \ (i=1,2)
\end{equation*}%
hold. Then $N$ has at least one fixed point $(u,v)\in K$ such that $|u|\leq
R_{1},$ $|v|\leq R_{2}$ and either $\Vert u\Vert \geq r_{1}$ or $\Vert
v\Vert \geq r_{2}.$
\end{thm}

\begin{proof}
Following the proof of Theorem \ref{thm:abstr1} and assuming that there
exists $\left( u,v\right) \in \partial U$ and $t\in \left( 0,1\right) $ with
$H(u,v,t)=\left( u,v\right) ,$ in case that $\left( u,v\right) \in A_{1},$
we have
\begin{equation*}
u=N_{1}\left( u,v\right) +t(h-N_{1}\left( u,v\right) )\geq N_{1}\left(
u,v\right) .
\end{equation*}%
Therefore, making use of the first inequality from (\ref{br1}), we obtain
\begin{equation*}
r_{1}=\left\Vert u\right\Vert \geq \left\Vert N_{1}\left( u,v\right)
\right\Vert >r_{1},
\end{equation*}
a contradiction. The same argument can be used if $\left( u,v\right) \in
A_{2}.$ Note that $h\notin \overline{U}$ because of the condition $%
r_{i}<\left\Vert h_{i}^{0}\right\Vert R_{i}$ $\left( i=1,2\right) .$
\end{proof}

We underline that, although the Harnack inequality is not used in the above
proof, it is essential to obtain the estimates from below of the type (\ref%
{br1}) in the applications.

\begin{rem}
We stress that the abstract results obtained in this section can be
generalized to the case of systems of more than two equations. The idea is
to consider the product space $X=\Pi _{i=1}^{n}X_{i}$ of the Banach spaces $%
X_{i},$ endowed with the norms $|\cdot|_{i},$ seminorms $\Vert \cdot \Vert _{i},$ and
the cones $K_{i}\subset K_{i}^{0}\subset X_{i}$ such that $\phi _{i}\in
K_{i}:=\mbox{$\left\{u\in K^0_i:\;u\geq \|u\|_i\chi_i\right\}$}$ for fixed $%
\chi _{i}$ and $\phi _{i}$ satisfying (\ref{eq:phi_i}), $i=1,2,...,n.$ In
this setting we are interested in the existence and localization of fixed
points of a given operator $N\colon K\rightarrow K,$ where $K=\Pi
_{i=1}^{n}K_{i}.$ For example, let us consider the sets
\begin{equation*}
C=\mbox{$\left\{u\in K:\;|u_1|_1\leq R_1,\ldots,|u_n|_n\leq R_n\right\}$},\
U=\mbox{$\left\{u\in C:\;\|u_1\|_1 < r_1,\ldots,\|u_n\|_n < r_n\right\}$}
\end{equation*}%
for given radii $r_{i},R_{i}$ with $r_{i}<\Vert \phi _{i}\Vert _{i}\Vert
\chi _{i}\Vert _{i}R_{i},\ $ $i=1,...,n.$ If
\begin{equation*}
\sup_{u\in C}|N_{i}(u)|_{i}\leq R_{i},\ \ i=1,2,...,n
\end{equation*}%
and
\begin{equation*}
\inf_{\substack{ (u,v)\in \overline{U}  \\ \Vert u_{i}\Vert _{i}=r_{i}}}%
\Vert N_{i}(u)\Vert _{i}\geq \frac{r_{i}}{\Vert \chi _{i}\Vert _{i}},\ \ \
i=1,2,...,n
\end{equation*}%
then $N$ has at least one fixed point in $C\setminus U.$

As a consequence, results analogous to ones obtained later in Section \ref%
{sect:appl}, can be established for systems with more than two differential
equations.
\end{rem}

\subsection{Case of isotone operators}

Let us now turn our attention to the case when $N$ satisfies a monotonicity
condition with respect to the order induced by cones $K_{i}^{0}.$ Precisely,
assume that the operator $N$ is \emph{isotone}, that is
\begin{equation*}
0\leq u\leq u^{\prime },\ 0\leq v\leq v^{\prime }\ \text{ implies }\
N_{1}(u,v)\leq N_{1}(u^{\prime },v^{\prime }),\ N_{2}(u,v)\leq
N_{2}(u^{\prime },v^{\prime }),
\end{equation*}%
and that condition~\eqref{br} is satisfied.

Let us examine the condition (\ref{6}). If $(u,v)\in C,$ $\Vert u\Vert
=r_{1},$ $\Vert v\Vert \leq r_{2},$ then $u\geq \Vert u\Vert \chi
_{1}=r_{1}\chi _{1}$ and $\Vert N_{1}(u,v)\Vert \geq \Vert N_{1}(r_{1}\chi
_{1},0)\Vert .$ Therefore the condition (\ref{6}) is implied by the simpler
one:
\begin{equation*}
\Vert N_{1}(r_{1}\chi _{1},0)\Vert \geq \frac{r_{1}}{\Vert \chi _{1}\Vert }%
,\ \ \Vert N_{2}(0,r_{2}\chi _{2})\Vert \geq \frac{r_{2}}{\Vert \chi
_{2}\Vert }.  \tag{\ref{6}'}
\end{equation*}%
Similarly, the condition (\ref{6'}) is implied by the following one:
\begin{equation*}
\Vert N_{1}(r_{1}\chi _{1},r_{2}\chi _{2})\Vert \geq \frac{r_{1}}{\Vert \chi
_{1}\Vert },\ \ \Vert N_{2}(r_{1}\chi _{1},r_{2}\chi _{2})\Vert \geq \frac{%
r_{2}}{\Vert \chi _{2}\Vert }.  \tag{\ref{6'}'}
\end{equation*}

Let us now examine the condition (\ref{4}). If $\left\vert u\right\vert \leq
R_{1}$ and $\left\vert v\right\vert \leq R_{2}$ then $u\leq R_{1}h_{1}^{0},$
$v\leq R_{2}h_{2}^{0}$ and $\left\vert N_{i}(u,v)\right\vert \leq \left\vert
N_{i}(R_{1}h_{1}^{0},R_{2}h_{2}^{0})\right\vert ,$ $i=1,2.$ This shows that
the condition
\begin{equation*}
|N_{i}(R_{1}h_{1}^{0},R_{2}h_{2}^{0})|\leq R_{i}\ \ \ \left( i=1,2\right)
\end{equation*}%
implies (\ref{4}).

\section{Applications to $\left( p,q\right) $-Laplacian systems}

\label{sect:appl}

\subsection{Existence results}

We now turn our attention back to the existence and localization of
nonnegative weak solutions of the $\left( p,q\right) $-Laplace system %
\eqref{eq:diff}. By a \textit{nonnegative} solution we mean a solution $%
\left( u,v\right) $ with $u\geq 0$ and $v\geq 0.$ A nonnegative solution $%
\left( u,v\right) $ is said to be \textit{nontrivial} if either $u\neq 0$ or
$v\neq 0,$ and is called \textit{positive} if both functions $u,v $ are
different from zero, equivalently, if $u>0$ and $v>0.$

In order to apply the abstract theorems from the previous section, we choose
$X_{1}=X_{2}=L^{\infty }(\Omega ),$ and $K_{1}^{0}=K_{2}^{0}=K^{0}:=\left\{
u\in L^{\infty }\left( \Omega \right) :u\geq 0\right\} .$ Thus the two norms
$\left\vert \cdot\right\vert _{1},\left\vert \cdot\right\vert _{2}$ coincide with
the usual norm of $L^{\infty }\left( \Omega \right) ,$ simply denoted by $%
\left\vert \cdot\right\vert .$ Let $S_{r}:L^{\infty }\left( \Omega \right)
\rightarrow C_{0}^{1}(\overline{\Omega })$ be the operator which assigns to
any $v\in L^{\infty }\left( \Omega \right) $ the unique weak solution $u$ of
the problem%
\begin{equation}
\begin{cases}
-\Delta _{r}u=v & \text{in }\Omega , \\
u=0 & \text{on }\partial \Omega ,%
\end{cases}
\label{eq:defin S}
\end{equation}%
i.e. $S_{r}=(-\Delta _{r})^{-1}.$ It is known (see \cite{b:AC}) that $S_{r}$
is well defined, completely continuous, isotone and positive. Also consider
the superposition operators $F,G\colon K^{0}\times K^{0}\rightarrow K^{0},$
\begin{equation*}
F(u,v)(x)=f(x,u(x),v(x)),\ \ G(u,v)(x)=g(x,u(x),v(x)),
\end{equation*}%
and define the operator $N=(N_{1},N_{2})$ by
\begin{equation*}
N_{1}(u,v)=S_{p}(F(u,v)),\ \ N_{2}(u,v)=S_{q}(G(u,v)).
\end{equation*}%
Note that a pair $(u,v)$ is a nonnegative weak solution of \eqref{eq:diff}
if and only if it is a fixed point of $N$ in $K^{0}\times K^{0}.$ In Remark~%
\ref{solutions-regular} we shall give some additional information on the
regularity of the solution.

We recall a local weak Harnack inequality for nonnegative $p$-superharmonic
functions due to N.S. Trudinger, see \cite[Theorem 1.2]{b:Trudinger}:

\begin{thm}
\label{lem:Trudinger} For each $s\in \lbrack 1,\frac{n\left( p-1\right) }{n-p%
})$ if $p<n$ or $s\in \lbrack 1,\infty )$ if $p\geq n$ and for each $\rho >0$
such that $B_{3\rho }\subset \Omega ,$ there exists a constant $C=C\left(
n,p,s\right) >0$ such that
\begin{equation*}
\inf_{B_{\rho }}u\geq C\rho ^{-\frac{n}{s}}\left( \int_{B_{2\rho
}}u^{s}dx\right) ^{\frac{1}{s}}
\end{equation*}%
for every nonnegative $p$-superharmonic function $u$ in $\Omega .$
\end{thm}

Following a reasoning based on finite cover by balls of any compact set (see
the proof of Corollary~1.2.9 in \cite{MR1919991} or Theorem~1.3 in \cite%
{b:RP_TMNA}), we obtain a variant of the Harnack inequality, stated in \cite%
{b:RP JFPTAA}, which plays a crucial role in our investigation.

\begin{thm}
\label{thm:Trudinger-ineq} For each $s$ as in Theorem \emph{\ref%
{lem:Trudinger}} and each compact $D\subset \Omega ,$ there is a constant $%
M=M\left( n,p,s,D,\Omega \right) >0$ such that
\begin{equation}
\inf_{D}u\geq M\left( \int_{D}u^{s}dx\right) ^{\frac{1}{s}}.
\label{eq:Trudinger-ineq}
\end{equation}%
for every nonnegative $p$-superharmonic function $u$ in $\Omega .$
\end{thm}

By means of the Trudinger-type inequality (\ref{eq:Trudinger-ineq}) we shall
obtain the existence, localization, and multiplicity of nonnegative
solutions for the problem (\ref{eq:diff}).

Let us now consider any two compact sets $D_{1},D_{2}\subset \Omega $ and
their characteristic functions $\chi _{D_{1}},$ $\chi _{D_{2}},$ which we
denote by $\chi _{1},$ $\chi _{2}.$ Since $p,q>2n/(n+1),$ by Theorem \ref%
{thm:Trudinger-ineq} we obtain that there are numbers $s_{1},s_{1}>1$ and
constants $M_{1},M_{1}>0$ such that
\begin{equation}
\inf_{D_{1}}u\geq \Vert u\Vert _{1}:=M_{1}\left(
\int_{D_{1}}u^{s_{1}}dx\right) ^{\frac{1}{s_{1}}},\ \ \inf_{D_{2}}v\geq
\Vert v\Vert _{2}:=M_{2}\left( \int_{D_{2}}u^{s_{2}}dx\right) ^{\frac{1}{%
s_{2}}}  \label{eq:Harnack-tmp}
\end{equation}%
for every nonnegative $p$-superharmonic function $u$ and $q$-superharmionic
fuction $v.$ Using the natural partial order in $L^{\infty }(\Omega )$ we
can rewrite the inequalities (\ref{eq:Harnack-tmp}) in the following way:
\begin{equation}
u\geq \Vert u\Vert _{1}\chi _{1},\ \ \ v\geq \Vert v\Vert _{2}\chi _{2}.
\label{eq:Harnack-ineq}
\end{equation}%
The nonnegativity of $f$ and $g$ gives that $N_{1}(u,v)$ and $N_{2}(u,v)$
are superharmonic for any $u,v\in K^{0}.$ Therefore we obtain in a similar
way as in the case of one equation studied in \cite{b:RP JFPTAA}, that
\begin{equation*}
N_{1}(K^{0}\times K^{0})\subset K_{1},\ \ N_{2}(K^{0}\times K^{0})\subset
K_{2},
\end{equation*}%
where
\begin{equation*}
K_{i}=\left\{ u\in K^{0}:u\geq \left\Vert u\right\Vert _{i}\chi _{i}\right\}
,\ \ \ \chi _{i}=\chi _{D_{i}}\ \ \left( i=1,2\right) .
\end{equation*}%
Therefore, $N:K_{1}\times K_{2}\rightarrow K_{1}\times K_{2}.$ In addition,
the complete continuity of $S_{p}$ and $S_{q}$ guarantees that $N$ is
completely continuous.

\begin{rem}
\label{rem:better-constants-appl} Note that the assumption (\ref{br}) is
satisfied in our context of differential equations, since the constant
function $h_{i}^{0}\equiv 1$ satisfies \eqref{br}. Indeed, since the
constant function $1$ is $p\ \left( q\right) $-superharmionic in $\Omega ,$
one has $h_{i}^{0}\in K_{i}.$ Also, the conditions $\left\vert
h_{i}^{0}\right\vert =1$ and $u\leq h_{i}^{0}$ for every $u\in K_{i}$ with $%
\left\vert u\right\vert \leq 1$ are trivially satisfied. Therefore we can
use Remark \ref{rem:betterconstant}. This yields that, instead of the
relation \eqref{5} between radii, we can consider the relation
\begin{equation}
r_{i}<\Vert 1\Vert _{i}R_{i}\ \ \left( i=1,2\right) ,  \label{5'}
\end{equation}%
and we can improve the constants $r_{i}/\Vert \chi _{i}\Vert $ in (\ref{6}),
(\ref{6-strong}), (\ref{6-stronger}), (\ref{6'}), (\ref{6''}), (\ref{6varrho}%
), replacing them by $r_{i}.$

Notice that if in our case, $\phi _{i}$ is chosen to be $\chi
_{i}/\left\vert \chi _{i}\right\vert =\chi _{i},$ then the relation (\ref{5}%
) becomes $r_{i}<\left\Vert 1\right\Vert _{i}^{2}R_{i},$ which is more
restrictive than (\ref{5'}).
\end{rem}

We can now state a result for the existence and localization of a
nonnegative solution of the system~\ref{eq:diff} with both nonzero
components (i.e. a positive solution).

\begin{thm}
\label{thm:diffSol} Let $r_{1},r_{2},R_{1},R_{2}$ satisfy \emph{(\ref{5'})}.
Assume that the following conditions hold:
\begin{equation}
\begin{split}
\frac{\displaystyle\max_{(x,\tau ,\sigma )\in \overline{\Omega }\times
\lbrack 0,R_{1}]\times \lbrack 0,R_{2}]}f(x,\tau ,\sigma )}{R_{1}^{p-1}}%
=:\alpha _{1}& \leq A_{p}:=\frac{1}{|S_{p}(1)|^{p-1}},\  \\
\frac{\displaystyle\max_{(x,\tau ,\sigma )\in \overline{\Omega }\times
\lbrack 0,R_{1}]\times \lbrack 0,R_{2}]}g(x,\tau ,\sigma )}{R_{2}^{q-1}}%
=:\alpha _{2}& \leq A_{q}:=\frac{1}{|S_{q}(1)|^{q-1}},
\end{split}
\label{eq:ind1}
\end{equation}%
and%
\begin{equation}
\begin{split}
\frac{\displaystyle\min_{(x,\tau ,\sigma )\in D_{1}\times \lbrack
r_{1},R_{1}]\times \lbrack 0,R_{2}]}f(x,\tau ,\sigma )}{r_{1}^{p-1}}=:\beta
_{1}& >B_{1,p}:=\frac{1}{\Vert S_{p}(\chi _{1})\Vert ^{p-1}}, \\
\frac{\displaystyle\min_{(x,\tau ,\sigma )\in D_{2}\times \lbrack
0,R_{1}]\times \lbrack r_{2},R_{2}]}g(x,\tau ,\sigma )}{r_{2}^{q-1}}=:\beta
_{2}& >B_{2,q}:=\frac{1}{\Vert S_{q}(\chi _{2})\Vert ^{q-1}}.
\end{split}
\label{eq:ind0}
\end{equation}%
Then there exists a positive solution $(u,v)$ of \emph{(\ref{eq:diff})} such
that $\left\vert u\right\vert \leq R_{1},$ $\left\vert v\right\vert \leq
R_{2},$ $\left\Vert u\right\Vert \geq r_{1}$ and $\left\Vert v\right\Vert
\geq r_{2}.$
\end{thm}

\begin{proof}
If $\left( u,v\right) \in C$ and $x\in \Omega ,$ then $0\leq u(x)\leq R_{1},$
$0\leq v(x)\leq R_{2}$ and therefore

\begin{equation*}
F(u,v)(x)=f(x,u(x),v(x))\leq \alpha _{1}R_{1}^{p-1}.
\end{equation*}%
By the isotonicity of $S_{p}$ and the monotonicity of $|\cdot|$ we have
\begin{equation*}
|N_{1}(u,v)|\leq |S_{p}(\alpha _{1}R_{1}^{p-1})|\leq \alpha
_{1}^{1/(p-1)}R_{1}|S_{p}(1)| \leq R_{1}.
\end{equation*}%
A similar estimate is true for $N_{2}.$ Hence, the condition \eqref{4bis}
from Theorem \ref{thm:abstr1'} is satisfied. Now, let $\left( u,v\right) \in
C$ be such that $\Vert u\Vert =r_{1}.$ For $x\in D_{1},$ we have $u(x)\geq
\Vert u\Vert =r_{1}$ and, as a consequence,
\begin{equation*}
F(u,v)(x)=f(x,u(x),v(x))\geq \beta _{1}r_{1}^{p-1}.
\end{equation*}%
By the isotonicity of $S_{p}$ and the monotonicity of $\Vert \cdot \Vert $ we
have
\begin{equation*}
\left\Vert N_{1}(u,v)\right\Vert \geq \left\Vert S_{p}(\beta
_{1}r_{1}^{p-1}\chi _{1})\right\Vert =\beta _{1}{}^{1/(p-1)}r_{1}\left\Vert
S_{p}(\chi _{1})\right\Vert >r_{1}.
\end{equation*}%
A similar estimate is true for $N_{2}.$ Hence, the condition (\ref{6'}) from
Theorem \ref{thm:abstr1'}, modified in accordance with Remark \ref%
{rem:betterconstant}, is satisfied.

The assertion now follows from Theorem \ref{thm:abstr1'} and Remark \ref%
{rem:better-constants-appl}.
\end{proof}

We now present the relationship between the constants that arise in Theorem~%
\ref{thm:diffSol}.

\begin{prop}
Let $\lambda _{1,p}$ be the first eigenvalue of the $p$-Laplace operator $%
-\Delta _{p}$ under the Dirichlet condition. The following relations hold
\begin{equation*}
A_{p}\leq \lambda _{1,p}\leq B_{1,p}\text{ \ and\ }A_{q}\leq \lambda
_{1,q}\leq B_{2,q}.
\end{equation*}
\end{prop}

\begin{proof}
Let $u>0$ be a positive eigenfunction corresponding to the eigenvalue $%
\lambda _{1,p}.$ Then $-\Delta _{p}u=\lambda _{1,p}u^{p-1}$ and since $S_{p}$
is homogeneous of degree $1/(p-1),$ we have
\begin{equation*}
u=S_{p}\left( \lambda _{1,p}u^{p-1}\right) =\lambda
_{1,p}^{1/(p-1)}S_{p}\left( u^{p-1}\right) .
\end{equation*}

a) Assume that $|u|=1.$ Then $0\leq u\leq 1$ and, by the isotonicity of $%
S_{p}$ and the monotonicity of the norm, we have
\begin{equation*}
1=|u|=\lambda _{1,p}^{1/(p-1)}\left\vert S_{p}\left( u^{p-1}\right)
\right\vert \leq \lambda _{1,p}^{1/(p-1)}|S_{p}(1)|.
\end{equation*}%
This implies that $A_{p}\leq \lambda _{1,p}.$

b) Now assume that $\Vert u\Vert =1.$ Then, by the Harnack inequality, we
have $u\geq \chi _{D_{1}}.$ By the isotonicity of $S_{p}$ and the
monotonicity of the semi-norm we obtain
\begin{equation*}
1=\Vert u\Vert =\lambda _{1,p}^{1/(p-1)}\left\Vert S_{p}\left(
u^{p-1}\right) \right\Vert \geq \lambda _{1,p}^{1/(p-1)}\Vert S_{p}(\chi
_{D_{1}})\Vert .
\end{equation*}%
This implies that $B_{1,p}\geq \lambda _{1,p}.$ Similarly one can prove that
$A_{q}\leq \lambda _{1,q}\leq B_{2,q}.$
\end{proof}

\begin{thm}
\label{thm:diffSol or} Let $r_{1},r_{2},R_{1},R_{2}$ satisfy \emph{(\ref{5'})%
} and let $\tilde{R}_{1}\leq R_{1}$, $\tilde{R}_{2}\leq R_{2}$. Assume that
the condition \emph{(\ref{eq:ind1})} is satisfied and that%
\begin{equation}
\begin{split}
\frac{\displaystyle\min_{(x,\tau ,\sigma )\in D_{1}\times {[0,\tilde{R}_{1}]}%
\times \lbrack 0,\tilde{R}_{2}]}f(x,\tau ,\sigma )}{r_{1}^{p-1}}=:\beta
_{1}^{\prime }& >B_{1,p}\ {\text{ \ or }} \\
\frac{\displaystyle\min_{(x,\tau ,\sigma )\in D_{2}\times \lbrack 0,\tilde{R}%
_{1}]\times {[0,\tilde{R}_{2}]}}g(x,\tau ,\sigma )}{r_{2}^{q-1}}=:\beta
_{2}^{\prime }& >B_{2,q}.
\end{split}
\label{eq:ind0 or}
\end{equation}%
Then there exists a nontrivial nonnegative solution $(u,v)$ of \emph{(\ref%
{eq:diff})} such that $\left\vert u\right\vert \leq R_{1},$ $\left\vert
v\right\vert \leq R_{2}$ and
\begin{equation}
\Vert u\Vert \geq r_{1}\text{ \ or\ }\Vert v\Vert \geq r_{2}\ \text{ or\ }%
|u|>\tilde{R}_{1}\ \text{ or\ }|v|>\tilde{R}_{2}.
\label{eq:conclusion ororor}
\end{equation}%
In particular, if $\tilde{R}_{1}=R_{1}$ and $\tilde{R}_{2}=R_{2},$ then
there exists a nontrivial nonnegative solution $(u,v)$ with either $%
\left\Vert u\right\Vert \geq r_{1}$ or $\left\Vert v\right\Vert \geq r_{2}.$
\end{thm}

\begin{proof}
As in the previous proof we know that the condition \eqref{4-4} from Theorem~%
\ref{thm:abstr1or A} is satisfied. Let
\begin{equation*}
(u,v)\in A:=%
\mbox{$\left\{(u,v):\;\| u\|<r_{1},\ \| v\|< r_{2},\ |u|\leq
\tilde R_1,\ |v|\leq \tilde R_2\right\}$}.
\end{equation*}%
Assume that the first inequality in (\ref{eq:ind0 or}) holds. Then for $x\in
D_{1}$ we have
\begin{equation*}
F(u,v)(x)=f(x,u(x),v(x))\geq \beta _{1}^{\prime }r_{1}^{p-1}
\end{equation*}%
and by the isotonicity of $S_{p}$ and the monotonicity of $\Vert \cdot \Vert ,$
\begin{equation*}
\left\Vert N_{1}(u,v)\right\Vert \geq \left\Vert S_{p}(\beta _{1}^{\prime
}r_{1}^{p-1}\chi _{1})\right\Vert =\beta _{1}^{\prime
1/(p-1)}r_{1}\left\Vert S_{p}(\chi _{1})\right\Vert >r_{1}.
\end{equation*}%
A similar estimate holds for $N_{2}.$ Hence, the condition (\ref{6 or A})
from Theorem \ref{thm:abstr1or A}, modified in accordance with Remark \ref%
{rem:betterconstant}, is satisfied. Therefore, we can apply Theorem \ref%
{thm:abstr1or A} and Remark \ref{rem:better-constants-appl}, and obtain a
solution $(u,v)\not\in A.$ Clearly, this is equivalent to (\ref%
{eq:conclusion ororor}).
\end{proof}

\begin{rem}
The importance of Theorem \ref{thm:diffSol or} consists in the fact that the
assumption (\ref{eq:ind0 or}) involves only one component of the system
nonlinearity $(f,g).$ Therefore, it allows different kinds of growth of $f$
and $g$ near the origin.
A similar remark also applies to the following theorem.
\end{rem}

\begin{thm}
\label{thm:diff3sols} Let $\rho _{i},\ r_{i},\ R_{i}$ satisfy $0<\rho
_{i}<r_{i}<\Vert 1\Vert _{i}R_{i}\ \ \left( i=1,2\right) .$ Assume that:
\begin{equation}
\frac{\displaystyle\max_{(x,\tau ,\sigma )\in \overline{\Omega }\times
\lbrack 0,R_{1}]\times \lbrack 0,R_{2}]}f(x,\tau ,\sigma )}{R_{1}^{p-1}}\leq
A_{p},\ \ \frac{\displaystyle\max_{(x,\tau ,\sigma )\in \overline{\Omega }%
\times \lbrack 0,R_{1}]\times \lbrack 0,R_{2}]}g(x,\tau ,\sigma )}{%
R_{2}^{q-1}}\leq A_{q},  \label{eq:R-ind1}
\end{equation}%
\begin{equation}
\frac{\displaystyle\max_{(x,\tau ,\sigma )\in \overline{\Omega }\times
\lbrack 0,\rho _{1}]\times \lbrack 0,\rho _{2}]}f(x,\tau ,\sigma )}{\rho
_{1}^{p-1}}<A_{p},\ \ \frac{\displaystyle\max_{(x,\tau ,\sigma )\in
\overline{\Omega }\times \lbrack 0,\rho _{1}]\times \lbrack 0,\rho
_{2}]}g(x,\tau ,\sigma )}{\rho _{2}^{q-1}}<A_{q},  \label{eq:rho-ind1}
\end{equation}%
\begin{equation}
\frac{\displaystyle\min_{(x,\tau ,\sigma )\in D_{1}\times \lbrack
r_{1},R_{1}]\times \lbrack 0,R_{2}]}f(x,\tau ,\sigma )}{r_{1}^{p-1}}%
>B_{1,p},\ \ \frac{\displaystyle\min_{(x,\tau ,\sigma )\in D_{2}\times
\lbrack 0,R_{1}]\times \lbrack r_{2},R_{2}]}g(x,\tau ,\sigma )}{r_{2}^{q-1}}%
>B_{2,q}.  \label{eq:r-ind0}
\end{equation}%
Then there exist three nonnegative solutions $(u_{i},v_{i})$ $\left(
i=1,2,3\right) $ of the system~\eqref{eq:diff} with%
\begin{eqnarray*}
\left\vert u_{1}\right\vert &<&\rho _{1},\ \ \left\vert v_{1}\right\vert
<\rho _{2}\ \ \left( \text{possibly zero solution}\right) ; \\
\left\Vert u_{2}\right\Vert &<&r_{1},\ \left\Vert v_{2}\right\Vert <r_{2};\
\left\vert u_{2}\right\vert >\rho _{1}\ \text{or }\left\vert
v_{2}\right\vert >\rho _{2}\ \ \left( \text{possibly one solution component
zero}\right) ; \\
\left\Vert u_{3}\right\Vert &>&r_{1},\ \left\Vert v_{3}\right\Vert >r_{2}\ \
\left( \text{both solution components nonzero}\right) .
\end{eqnarray*}%
Moreover, having given numbers $0<\varrho _{i}<\Vert 1\Vert \rho _{i}$ $%
\left( i=1,2\right) ,$

\emph{(i)} if
\begin{equation}
\frac{\displaystyle\min_{(x,\tau ,\sigma )\in D_{1}\times \lbrack \varrho
_{1},\rho _{1}]\times \lbrack 0,\rho _{2}]}f(x,\tau ,\sigma )}{\varrho
_{1}^{p-1}}>B_{1,p}\ \text{ and}\ \ \frac{\displaystyle\min_{(x,\tau ,\sigma
)\in D_{2}\times \lbrack 0,\rho _{1}]\times \lbrack \varrho _{2},\rho
_{2}]}g(x,\tau ,\sigma )}{\varrho _{2}^{q-1}}>B_{2,q}  \label{eq:ind0-u_1}
\end{equation}%
then $\Vert u_{1}\Vert \geq \varrho _{1}$ and $\Vert v_{1}\Vert \geq \varrho
_{2};$

\emph{(ii)} if
\begin{equation}
\frac{\displaystyle\min_{(x,\tau ,\sigma )\in D_{1}\times \lbrack 0,\tilde{%
\rho}_{1}]\times \lbrack 0,\tilde{\rho}_{2}]}f(x,\tau ,\sigma )}{\varrho
_{1}^{p-1}}>B_{1,p}\ \text{ or }\ \ \frac{\displaystyle\min_{(x,\tau ,\sigma
)\in D_{2}\times \lbrack 0,\tilde{\rho}_{1}]\times \lbrack 0,\tilde{\rho}%
_{2}]}g(x,\tau ,\sigma )}{\varrho _{2}^{q-1}}>B_{2,q}  \label{eq:ind0 or-u_1}
\end{equation}%
for some $\tilde{\rho}_{1}\leq \rho _{1},$ $\tilde{\rho}_{2}\leq \rho _{2},$
then $\Vert u_{1}\Vert \geq \varrho _{1}$ or $\Vert v_{1}\Vert \geq \varrho
_{2}$ or $|u_{1}|>\tilde{\rho}_{1}$ or $|v_{1}|>\tilde{\rho}_{2}.$
\end{thm}

\begin{proof}
Observe that \eqref{eq:seminorm-cont} implies that the constants $c_i$ that
occur in the statement of the Theorem~\ref{thm:abstr3sols} are equal to 1.

The assumptions (\ref{eq:R-ind1}), (\ref{eq:rho-ind1}) and (\ref{eq:r-ind0})
imply the conditions (\ref{4tris}), (\ref{eq:deg1-rho}), (\ref{6'>}),
modified according to Remark~\ref{rem:betterconstant}. Hence, the first part
of the assertion follows from Theorem \ref{thm:abstr3sols'}, combined with
Remark~\ref{rem:better-constants-appl}.

The second part follows, in a similar way, from Theorem \ref%
{thm:abstr3sols''}, combined with Remark~\ref{rem:better-constants-appl}.
\end{proof}

\begin{rem}
If the functions $f,g$ do not depend on $x$ and are monotone (increasing or
decreasing) in each of the two variables $u,v,$ then the minima and maxima
in all inequalities from Theorems \ref{thm:diffSol}-\ref{thm:diff3sols} are
reached on some boundary point of the corresponding rectangular domain.
Combining the two kinds of monotonicity on the two variables $u,v,$ for our
functions $f\left( u,v\right) ,g\left( u,v\right) ,$ we see that sixteen
cases are possible. For example, in the case of Theorem \ref{thm:diffSol},
we may have the following situations:

(1) $f,g$ increasing in both variables; then the conditions (\ref{eq:ind1})
and (\ref{eq:ind0}) read as%
\begin{equation*}
\begin{array}{ll}
f\left( R_{1},R_{2}\right) \leq A_{p}R_{1}^{p-1} & \ g\left(
R_{1},R_{2}\right) \leq A_{q}R_{2}^{q-1} \\
f\left( r_{1},0\right) >B_{1,p}r_{1}^{p-1} & \ g\left( 0,r_{2}\right)
>B_{2,q}r_{2}^{q-1};%
\end{array}%
\end{equation*}

(2) $f$ increasing in $u$ and decreasing in $v,$ and $g$ increasing in both $%
u,v;$ then the conditions (\ref{eq:ind1}) and (\ref{eq:ind0}) become%
\begin{equation*}
\begin{array}{ll}
f\left( R_{1},0\right) \leq A_{p}R_{1}^{p-1} & \ g\left( R_{1},R_{2}\right)
\leq A_{q}R_{2}^{q-1} \\
f\left( r_{1},R_{2}\right) >B_{1,p}r_{1}^{p-1} & \ g\left( 0,r_{2}\right)
>B_{2,q}r_{2}^{q-1}.%
\end{array}%
\end{equation*}
\end{rem}

\begin{rem}
\label{solutions-regular} It is worth mentioning that the solutions $(u,v)$
of the system (\ref{eq:diff}) are continuously differentiable. This follows
from the fact that $S_{r}(L^{\infty }\left( \Omega \right) )\subset C^{1}(%
\overline{\Omega }),$ $r>1$ (see \cite{b:AC}). Moreover, as a consequence of
the Harnack inequality (see \cite{b:Trudinger}), if $u\geq 0$ and $u\neq 0,$
then
\begin{equation*}
u(x)>0\ \text{ for }\ x\in \Omega
\end{equation*}%
(the same for $v$). Indeed, the Harnack inequality implies that the set $%
\{x\in \Omega :u\left( x\right) =0\}$ is open. Being also closed in $\Omega
, $ it is equal either to $\emptyset $ or to $\Omega .$ The latter is
excluded by the assumption $u\neq 0.$
\end{rem}

\subsection{Non-existence results}

We now present some sufficient conditions for the non-existence of positive
solutions of the system~\eqref{eq:diff}.

\begin{thm}
Assume that $(u,v)$ is a nonnegative solution of the system~\eqref{eq:diff}.
If one of the following conditions holds:
\begin{align}
f(x,u,v)& <\lambda _{1,p}u^{p-1},\ \text{for every}\ (x,u,v)\in \Omega
\times (0,\infty )\times \lbrack 0,\infty ),  \label{eq:nonexistence1} \\
f(x,u,v)& >\lambda _{1,p}u^{p-1},\ \text{for every}\ (x,u,v)\in \Omega
\times (0,\infty )\times \lbrack 0,\infty ),  \label{eq:nonexistence2} \\
f(x,u,v)& >B_{1,p}u^{p-1},\ \text{for every}\ (x,u,v)\in D_{1}\times
(0,\infty )\times \lbrack 0,\infty ),  \label{eq:nonexistence3}
\end{align}%
then $u=0.$ Similarly, if one of the following conditions holds:
\begin{align}
g(x,u,v)& <\lambda _{1,q}v^{q-1},\ \text{for every}\ (x,u,v)\in \Omega
\times \lbrack 0,\infty )\times (0,\infty ),  \label{eq:nonexistence1'} \\
g(x,u,v)& >\lambda _{1,q}v^{q-1},\ \text{for every}\ (x,u,v)\in \Omega
\times \lbrack 0,\infty )\times (0,\infty ),  \label{eq:nonexistence2'} \\
g(x,u,v)& >B_{2,q}v^{q-1},\ \text{for every}\ (x,u,v)\in D_{2}\times \lbrack
0,\infty )\times (0,\infty ),  \label{eq:nonexistence3'}
\end{align}%
then $v=0.$
\end{thm}

\begin{proof}
Let us observe that $-\Delta _{p}u=f(x,u,v)\geq 0.$ Suppose on the contrary
that $u\neq 0.$ Then $u(x)>0$ for $x\in \Omega $ (see Remark \ref%
{solutions-regular}).

Assume that the inequality (\ref{eq:nonexistence1}) holds. Then $%
f(x,u(x),v(x))<\lambda_{1,p}u^{p-1}(x)$ almost everywhere in $\Omega$ and
consequently
\begin{equation*}
\lambda_{1,p}\int_\Omega u^p\leq \int_\Omega |\nabla u|^p=\int_\Omega
f(x,u,v)\cdot u<\lambda_{1,p}\int_\Omega u^{p-1}u,
\end{equation*}
a contradiction.

Assume now that (\ref{eq:nonexistence2}) holds. Then
\begin{equation}
-\Delta _{p}u=\lambda _{1,p}u^{p-1}+h,  \label{eq:not possible}
\end{equation}%
where the function $h(x):=f(x,u(x),v(x))-\lambda _{1,p}u^{p-1}(x)$ is
positive almost everywhere in $\Omega $ and is obviously in the space $%
L^{\infty }(\Omega ).$ This contradicts the fact from \cite{MR1354715},
which states that (\ref{eq:not possible}) has no solutions.

Assume now that the inequality (\ref{eq:nonexistence3}) holds. Put $r:=\Vert
u\Vert .$ Then, since $u\in K_{1},$ we have that $u(x)\geq r$ for $x\in
D_{1}.$ Therefore
\begin{equation*}
F(u,v)(x)=f(x,u(x),v(x))>B_{1,p}u(x)^{p-1}\geq B_{1,p}r^{p-1}.
\end{equation*}%
By the continuity of $f,$ $u$ and $v$ (see Remark \ref{solutions-regular})
and the compactness of $D_{1}$ we deduce that $F(u,v)(x)\geq
B_{1,p}r_{1}^{p-1}$ for some $r_{1}>r$ and all $x\in D_{1}.$ This implies
that $F(u,v)\geq B_{1,p}r_{1}^{p-1}\chi _{1}.$ Then, by the isotonicity of $%
S_{p},$ the monotonicity of the seminorm and the definition of $B_{1,p},$ we
derive
\begin{equation*}
r=\Vert u\Vert =\Vert N_{1}(u,v)\Vert =\Vert S_{p}F(u,v)\Vert \geq \Vert
S_{p}(B_{1,p}r_{1}^{p-1})\Vert =B_{1,p}^{1/(p-1)}r_{1}\Vert S_{p}(\chi
_{1})\Vert =r_{1},
\end{equation*}%
a contradiction.

The second assertion can be proved analogously.
\end{proof}

\begin{cor}
\emph{(i)} If one of the inequalities \emph{(\ref{eq:nonexistence1})-(\ref%
{eq:nonexistence3'})} holds, then there are no positive solutions of the
system \emph{(\ref{eq:diff})}.

\emph{(ii)} If one of the inequalities \emph{(\ref{eq:nonexistence1})-(\ref%
{eq:nonexistence3})} holds and one of the inequalities \emph{(\ref%
{eq:nonexistence1'})-(\ref{eq:nonexistence3'})} holds, then there are no
nontrivial nonnegative solutions of the system \emph{(\ref{eq:diff})}.
\end{cor}

\subsection{Example}

Let%
\begin{equation*}
f(x,u,v)=\varphi (x,u,v)\cdot \frac{u^{2}}{4+u^{3}},\quad g(x,u,v)=\psi
(x,u,v)\cdot \arctan ^{2}v,
\end{equation*}%
where $\varphi $ and $\psi $ are continuous functions such that
\begin{equation*}
a\leq \varphi (x,u,v)\leq b,\ \ \ c\leq \psi (x,u,v)\leq d,\
\end{equation*}%
for every$\ (x,u,v)\in \Omega \times \lbrack 0,+\infty )\times \lbrack
0,+\infty )$ and some fixed constants $a,b,c,d\in (0,+\infty ).$

\begin{prop}
Under the above assumptions, there exists a constant $\lambda _{0}$
(depending on $a,b,c,d$ and $\Omega $) such that for each $\lambda >\lambda
_{0},$ the problem%
\begin{equation}
\left\{
\begin{array}{ll}
-\Delta u=\lambda f(x,u,v) & \text{in }\Omega , \\
-\Delta v=\lambda g(x,u,v) & \text{in }\Omega , \\
u,v=0 & \text{on }\partial \Omega ,%
\end{array}%
\right.  \label{eq:diff exa}
\end{equation}%
has at least two nontrivial nonnegative solutions.
\end{prop}

\begin{proof}
We shall apply Theorem \ref{thm:diff3sols}. To this end, consider any
compact subset $D\subset \Omega $ and put $D_{1}:=D,$ $D_{2}:=D.$ In order
to simplify our notation we denote
\begin{equation*}
A:=A_{2},\ B:=B_{1,2}=B_{2,2},\ \Phi (x)=x^{2}/(4+x^{3})\ \ \text{and\ \ }%
\Psi (x)=\arctan ^{2}x.
\end{equation*}%
Note that $\Phi (x)\leq l_{1}:=1/3,$ $\Psi (x)\leq l_{2}:=\pi ^{2}/4,$ $\Psi
$ is increasing, while $\Phi $ is increasing in $[0,2]$ and decreasing in $%
[2,\infty ).$ Therefore, if $\rho _{1}<2$ then the conditions %
\eqref{eq:R-ind1}, \eqref{eq:rho-ind1} and \eqref{eq:r-ind0} are implied by
the following ones:%
\begin{equation}
\lambda l_{1}b\leq AR_{1},\quad \lambda l_{2}d\leq AR_{2},  \label{eq:exa1}
\end{equation}%
\begin{equation}
\frac{\Phi (\rho _{1})}{\rho _{1}}<\frac{A}{\lambda b},\quad \frac{\Psi
(\rho _{2})}{\rho _{2}}<\frac{A}{\lambda d},  \label{eq:exa2}
\end{equation}%
\begin{equation}
\lambda a\min \{\Phi (r_{1}),\Phi (R_{1})\}>Br_{1},\quad \frac{\Psi (r_{2})}{%
r_{2}}>\frac{B}{\lambda c}.  \label{eq:exa3}
\end{equation}%
If we put
\begin{equation*}
R_{1}(\lambda ):=\lambda l_{1}b/A\ \ \ \text{and\ \ \ }R_{2}(\lambda
):=\lambda l_{2}d/A,
\end{equation*}%
then the condition \eqref{eq:exa1} is satisfied and the condition %
\eqref{eq:exa3} has the form
\begin{equation}
\frac{R_{1}(\lambda )}{r_{1}}\min \{\Phi (r_{1}),\ \ \Phi (R_{1})\}>l_{1}%
\frac{B}{A}\frac{b}{a}=:\gamma ,\ \ \frac{\Psi (r_{2})}{r_{2}}>\frac{B}{%
\lambda c}.  \label{eq:exa3'}
\end{equation}%
Fix $r_{2}>0.$ There exists $\lambda _{1}$ such that the second inequality
from \eqref{eq:exa3'}, as well as $r_{2}<\Vert 1\Vert R_{2}(\lambda ),$ are
satisfied for $\lambda >\lambda _{1}.$ Let us put
\begin{equation*}
r_{1}=r_{1}(\lambda )=1/\sqrt{R_{1}(\lambda )}.
\end{equation*}%
Then there exists $\lambda _{2}\geq \lambda _{1}$ such that $r_{1}(\lambda
)<\Vert 1\Vert R_{1}(\lambda )$ for $\lambda >\lambda _{2}.$ The first
inequality from \eqref{eq:exa3'} becomes
\begin{equation}
\min \left\{ \frac{R_{1}(\lambda )\sqrt{R_{1}(\lambda )}}{1+4R_{1}(\lambda )%
\sqrt{R_{1}(\lambda )}},\ \frac{R_{1}^{3}(\lambda )}{4+R_{1}^{3}(\lambda )}%
\right\} >\frac{\gamma }{\sqrt{R_{1}(\lambda )}}.  \label{eq:exa3''}
\end{equation}%
If $\lambda \rightarrow \infty ,$ then the left-hand side of %
\eqref{eq:exa3''} tends to $1/4$ and the right-hand side of \eqref{eq:exa3''}
tends to $0.$ Therefore, there exists $\lambda _{0}\geq \lambda _{2}$ such
that \eqref{eq:exa3''} is satisfied for $\lambda >\lambda _{0}.$ Having
defined $r_{i}$ and $R_{i}$ for $\lambda >\lambda _{0},$ let us choose $\rho
_{1}<r_{1}$ and $\rho _{2}<r_{2}$ that satisfy \eqref{eq:exa2}. This is
possible because $\Phi (x)=o(x)$ and $\Psi (x)=o(x)$ in $0.$ By Theorem \ref%
{thm:diff3sols} we obtain the existence of at least two nontrivial
nonnegative solutions to \eqref{eq:diff exa} for $\lambda >\lambda _{0}.$

Let us observe that for a fixed $\lambda >\lambda _{0},$ the nontrivial
nonnegative solutions $(u_{1},v_{1})$ and $(u_{2},v_{2})$ given by Theorem %
\ref{thm:diff3sols}, satisfy
\begin{eqnarray*}
|u_{i}| &\leq &R_{1}(\lambda ),\ |v_{i}|\leq R_{2}(\lambda )\ \ \ \left(
i=1,2\right) ,\ \Vert u_{1}\Vert >r_{1}(\lambda ),\Vert v_{1}\Vert >r_{2}, \\
\Vert u_{2}\Vert &<&r_{1}(\lambda ),\ \Vert v_{2}\Vert <r_{2}\ \ \text{and\
\ }|u_{2}|>\rho _{1}(\lambda )\ \ \text{or\ \ }|v_{2}|>\rho _{2}(\lambda ).
\end{eqnarray*}%
Note that both components of the first solution are positive.
\end{proof}

\section*{Acknowledgments}
R. Precup was supported by a grant of the Romanian National
Authority for Scientific Research, CNCS -- UEFISCDI, project number
PN-II-ID-PCE-2011-3-0094. This paper was written during the postdoctoral
stage of M. Maciejewski at the University of Calabria, supported by a
research fellowship within the project ``Enhancing Educational Potential of
Nicolaus Copernicus University in the Disciplines of Mathematical and
Natural Sciences'' (Project no. POKL.04.01.01-00-081/10). A support from the Chair of Nonlinear Mathematical Analysis and Topology is gratefully acknowledged.

\end{document}